\documentclass{amsart}
\usepackage[utf8]{inputenc}
\usepackage[T1]{fontenc}
\usepackage{amsfonts}
\usepackage{amsmath}
\usepackage{amssymb}
\usepackage{amsthm}
\usepackage{a4wide}
\usepackage{graphicx}
\usepackage{mathrsfs}
\usepackage{hyperref} 

\DeclareMathAlphabet{\mathpzc}{OT1}{pzc}{m}{it}

\newtheorem{theorem}{Theorem}
\newtheorem{proposition}[theorem]{Proposition}
\newtheorem{corollary}[theorem]{Corollary}
\newtheorem{lemma}[theorem]{Lemma}
\theoremstyle{definition}
\newtheorem{definition}{Definition}
\newtheorem{example}{Example}
\newtheorem{remark}{Remark}

\title{On the real differential of a slice regular function} 

\author{Amedeo Altavilla} 

\thanks{Dipartimento di Ingegneria Industriale e Scienze Matematiche 
Universit\`a Politecnica delle Marche, Via Brecce Bianche,
I-60131 Ancona, Italy, amedeoaltavilla@gmail.com}


\begin{document}


\begin{abstract}
In this paper we show that the real differential of any injective slice regular function is  everywhere invertible.
The result is a generalization of a theorem proved by G. Gentili, S. Salamon and C. Stoppato and it is obtained thanks, in 
particular, to some new information regarding the first coefficients of a certain polynomial expansion for slice regular functions
(called \textit{spherical expansion}), and to a new general result which says that the slice derivative of any injective slice 
regular function 
is different from zero. A useful tool proven in this paper is a new formula that relates slice and spherical derivatives of a slice regular function.
Given a slice regular function, part of its singular set is described as the union of surfaces on which it results to be constant.
\end{abstract}
\maketitle


\section{Introduction}

In \cite{stoppato} and \cite{gensalsto}, the authors started an interesting investigation about the real differential 
of a slice regular function (see also \cite{genstostru}, Chapter 8.5),  which gave interesting results both from theoretical and applicative point of views.
Let us firstly describe in few words the actors of this story which will be properly formalized in the next section.

Denoting by $\mathbb{H}$ the real algebra of quaternions and by 
$\mathbb{S}\subset\mathbb{H}$ the subset of imaginary units,
\begin{equation*}
\mathbb{S}:=\{q\in\mathbb{H}\,|\,q^{2}=-1\},
\end{equation*}
we can write any quaternion as $x=\alpha+I\beta$,
where $\alpha$ and $\beta$ are real numbers and $I\in\mathbb{S}$.

A \textit{slice function} $f:\Omega\subseteq\mathbb{H}\rightarrow\mathbb{H}$
is a quaternionic function of one quaternionic variable that is $\mathbb{H}$-left
affine with respect to the imaginary unit, i.e.
 such that, 
for each $x=\alpha+I\beta\in\Omega$, it holds,
\begin{equation*}
f(x)=F_{1}(\alpha, \beta)+IF_{2}(\alpha, \beta),
\end{equation*}
where, $F_{1}$ and $F_{2}$ satisfy an additional technical requirement.

A \textit{slice regular function} is a slice function $f$ such that, for any 
$I\in\mathbb{S}$ its restriction to the complex line
$\mathbb{C}_{I}:=span_{\mathbb{R}}\{1,I\}\subset\mathbb{H}$ is a holomorphic map.

The theory standing on this notion of regularity, introduced by Cullen in \cite{cullen}, is rapidly growing in the last years, thanks firstly 
to the authors of \cite{colgensabstru,gentilistruppa1,gentilistruppa}, who have set down the groundwork.

The main purpose of this paper is to \textit{extend}, by removing the hypothesis concerning the domain, the next theorem stated in \cite{gensalsto}.

\begin{theorem}[\cite{gensalsto}, Corollary 3.10]\label{teoremadaestendere}
 Let $f:\Omega\rightarrow \mathbb{H}$ be an injective slice regular function with $\Omega\cap\mathbb{R}\neq\emptyset$. 
 Then its real differential is everywhere invertible.
\end{theorem}

Many results about
slice regular functions defined over domains that intersects the real line, does not extend in an automatic way to functions defined over domains that do
not have real points (see e.g. \cite{altavilla,ghiloniperotti,ghiloniperotti2,ghiloniperotti3}). 

Even in this case, the proof of Theorem \ref{teoremadaestendere} contained in \cite{gensalsto} can not
be adapted to our setting.
To obtain the goal, we apply results from \cite{gensalsto,ghiloniperotti3,stoppato}, taking into account the mentioned difference.

In particular, we firstly realize that a part of the theory can be formalized in a new way considering the \textit{slice factor} of the real differential
of a slice function: we introduce here the concept of \textit{slice differential} of a slice function. 
After that, we remember the notion of \textit{spherical analyticity} introduced in \cite{stoppato} for functions with domain intersecting the real axis, and in
\cite{ghiloniperotti3} for functions defined over any domain (in a more general context). We give some new information about
the first coefficients of the spherical expansion of a slice regular function, showing a new way to compute slice derivatives (see formula \ref{formula}).
Finally, starting from some results about the rank of the real differential of a slice regular function, we extend Theorem \ref{teoremadaestendere}, using,
moreover, a new proposition that generalizes, in our context, a classical theorem of complex analysis (see Theorem \ref{thmslder}).
While describing these materials we show that a slice regular function can
be constant either globally or on sets of real dimension two. 

The structure of the present work is the following.

In Section 2 we state the main preliminary results needed for the reader to understand the theory. 
In this section new results regarding the 
sets on which a slice regular function can be constant are showed.
For the non-original material, we mention as general references for this part,
the two books \cite{colsabstru,genstostru} and the paper \cite{ghiloniperotti}.
At the end we introduce the concept of slice differential of a slice function. 

Section 3 is divided into 2 subsections: in the first part we remember the main results about spherical 
analyticity and we discuss about the first coefficients of this expansion; in second part, 
we deal with the rank of the real differential of a slice regular function and prove the main theorem.

Most of the results presented in this paper are contained in the Ph.D. 
thesis of the author \cite{altavillaPhD}.

\section{Preliminary results}
We introduce in this section
some notion involved with the study of slice regular functions developed by R. Ghiloni and A. Perotti (see \cite{ghiloniperotti,ghiloniperotti2,ghiloniperotti3}).

In the real algebra of quaternions $\mathbb{H}$ we will denote with $x^c$ the usual conjugation, i.e.: if $x=x_0+ix_1+jx_2+kx_3\in\mathbb{H}$, then $x^c=x_0-ix_1-jx_2-kx_3$.
Let $\mathbb{H}_\mathbb{C}:=\mathbb{H}\otimes_\mathbb{R} \mathbb{C}$ be the complexification of $\mathbb{H}$. An element $x$ in $\mathbb{H}_\mathbb{C}$ will be
of the form $x=p+\sqrt{-1}q$, where $p$ and $q$ are quaternions. 
The space $\mathbb{H}_{\mathbb{C}}$ is a complex alternative algebra
with a unity 
with respect to the product defined by the formula
\begin{equation*}
(x+\sqrt{-1}y)(z+\sqrt{-1}w):=xz-yw+\sqrt{-1}(xw+yz).
\end{equation*}
In $\mathbb{H}_\mathbb{C}$ we define two commuting anti-involution acting on the element $x$ in the 
following ways:
\begin{itemize}
 \item $x^c=(p+\sqrt{-1}q)^c=p^c+\sqrt{-1}q^c$,
 \item $\overline{x}=\overline{(p+\sqrt{-1}q)}=p-\sqrt{-1}q$
\end{itemize}

Let now $D$ be a connected open set in $\mathbb{C}$, we remember the following definitions from \cite{ghiloniperotti}.
\begin{definition}[\cite{ghiloniperotti}]\label{stem}
 A function $F:D\rightarrow \mathbb{H}_\mathbb{C}$ is called a \textit{stem function} on $D$ if it is complex intrinsic, i.e.: if the condition $F(\overline{z})=\overline{F(z)}$ holds 
 for each $z\in D$ such that $\overline{z}\in D$. 
 Moreover we say that $F=F_1+\sqrt{-1}F_2$ has a certain regularity (e.g.: $\mathcal{C}^0$, $\mathcal{C}^k$, $\mathcal{C}^\omega$, etc),
 if the two components $F_1$ and $F_2$ have that regularity. 
\end{definition}

The last definition means that if $D$ is symmetric with respect to the real axis and 
$F_1, F_2: D\rightarrow \mathbb{H}$ are the quaternionic components of 
$F=F_1+\sqrt{-1}F_2$ then $F_1$ is even with respect to the imaginary
part of $z$ (i.e.: $F_1(\overline{z})=F(z)$), while $F_2$ is odd 
(i.e.: $F_2(\overline{z})=-F_2(z)$). Thanks to this fact, there is no loss of generality in
requiring $D\subset\mathbb{C}$ to be symmetric with respect to the real axis. 

\begin{definition}[\cite{ghiloniperotti}]
 Given any set $D\subset \mathbb{C}$ we define the \textit{circularization} of $D$ in $\mathbb{H}$ as the subset of $\mathbb{H}$ defined by:
 \begin{equation*}
  \Omega_D:=\{\alpha+ J\beta\in\mathbb{H}\,|\,\alpha+i\beta\in D\, ,\, J\in\mathbb{S}\}.
 \end{equation*}
 Moreover any set of this kind will be called a \textit{circular set}.
\end{definition}
\begin{remark}
If $D$ is a domain in $\mathbb{C}$ such that $D\cap\mathbb{R}\neq\emptyset$, then $\Omega_{D}$
is also called \textit{slice domain} (see \cite{gensalsto}, Definition 1.23).
\end{remark}
We will use the following notation: let $D$ be a subset of the complex plane $\mathbb{C}$, then we will denote with $D_J$ and  $D_J^+$ the following
sets:
\begin{equation*}
 D_J:=\Omega_D\cap \mathbb{C}_J,\quad D_J^+:=\Omega_D\cap \mathbb{C}_J^+,\quad\forall J\in\mathbb{S},
\end{equation*}
where $\mathbb{C}_J:=\{\alpha+J\beta\in \mathbb{H}\,\mid\,\alpha+i\beta\in \mathbb{C},\,J\in\mathbb{S}\}$ and $\mathbb{C}_J^+:=\{\alpha+J\beta\in \mathbb{C}_J\,|\,\beta\geq 0\}$. 
Moreover, we will call sets of the form $D_{J}$ and $D_{J}^{+}$ a \textit{slice}
and \textit{semi-slice} of $\Omega_{D}$, respectively.
\begin{remark}
Let $D\subset \mathbb{C}$ be any set. We denote $D^+:=\{z\in D\,|\, Im(z)>0\}$.
If $D\cap\mathbb{R}=\emptyset$, then $\Omega_{D}\simeq D^+\times \mathbb{S}$
and so $D_{J}\simeq D^+\times \{-J,J\}$, while $D_{J}^{+}\simeq D^+\times\{J\}$.
\end{remark}

\begin{definition}[\cite{ghiloniperotti}]
 A function $f:\Omega_D\rightarrow\mathbb{H}$ is called a \textit{(left) slice function} if it is induced by a stem function $F=F_1+\sqrt{-1}F_2$ on 
 $D$, $f=\mathcal{I}(F)$, in the following way:
 \begin{equation*}
  f(\alpha+ J\beta):=F_1(\alpha+i\beta)+JF_2(\alpha+i\beta),\quad \forall\, x=\alpha+J\beta\in\Omega_D.
 \end{equation*}
\end{definition}
We will denote by $\mathcal{S}(\Omega_D)$ and by $\mathcal{S}^k(\Omega_D)$, for any $k\in\mathbb{N}\cup\{\infty\}$, 
the real vector spaces and right $\mathbb{H}$-module of slice functions on $\Omega_D$ induced 
by continuous and of class $\mathcal{C}^{k}$ stem functions, respectively.
Thanks to Definition \ref{stem}, any slice function is well defined. Indeed, if $D$ is symmetric with respect to the real axis and 
$f=\mathcal{I}(F):\Omega_D\rightarrow \mathbb{H}$ is a slice function induced by
$F$, then $f(\alpha+(-J)(-\beta))=F_1(\alpha+i(-\beta))-JF_2(\alpha+i(-\beta))=F_1(\alpha+i\beta)-J(-F_2(\alpha+i\beta))=
F_1(\alpha+i\beta)+JF_2(\alpha+i\beta)=f(\alpha+J\beta)$.

For slice functions we have the following representation theorem. It says that if we know the values of a slice function over two different semi-slices 
then we can reconstruct the whole function. This is not a surprising result having
in mind the ``affine nature'' of a slice function with respect to the imaginary unit.
The precise statement is the following one.
\begin{theorem}[\cite{ghiloniperotti}, Proposition 6]\label{representationtheorem}
Let $f$ be a slice function on $\Omega_D$. If $J\neq K\in\mathbb{S}$ then, for every $x=\alpha+I\beta\in\Omega_D$, the following formula holds:
\begin{equation*}
 f(x)=(I-K)(J-K)^{-1}f(\alpha+J\beta)-(I-J)(J-K)^{-1}f(\alpha+K\beta).
\end{equation*}
In particular, for $K=-J$, we get the 
formula
\begin{equation*}
 f(x)=\frac{1}{2}(f(\alpha+J\beta)+f(\alpha-J\beta)-IJ(f(\alpha+J\beta)-f(\alpha-J\beta))).
\end{equation*}
\end{theorem}

Representation formulas for quaternionic slice regular functions appeared in \cite{colgensab,colgensabstru}, while the case of continuous slice functions can be found
in \cite{ghiloniperotti}.

\begin{definition}[\cite{ghiloniperotti}]
 Given a slice function $f$, we define its spherical derivative in $x\in\Omega_{D}\setminus \mathbb{R}$ as,
\begin{equation*}
\partial_{s}f(x):=\frac{1}{2}Im(x)^{-1}(f(x)-f(x^{c})).
\end{equation*}
\end{definition}

\begin{remark}\label{spherical}
We have that $\partial_{s}f=\mathcal{I}(\frac{F_{2}(z)}{Im(z)})$ on $\Omega_{D}\setminus\mathbb{R}$.
Given $x=\alpha+J\beta\in\Omega_D$, the spherical derivative is constant on the sphere 
\begin{equation*}
\mathbb{S}_{x}=\{y\in\mathbb{H}\,|\,y=\alpha +I\beta ,I\in\mathbb{S}\}.
\end{equation*}
Moreover, $\partial_{s}f=0$ if and only if $f$ is constant on $\mathbb{S}_{x}$, in other terms:
\begin{equation*}
\partial_{s}(\partial_{s}(f))=0.
\end{equation*}
If $\Omega_{D}\cap \mathbb{R}\neq\emptyset$, under some mild regularity hypothesis on $F$ (see \cite{ghiloniperotti}, Proposition 7 for more details), $\partial_{s}f$ can be 
extended continuously as a slice function on $\Omega_{D}$. In particular this is true if the stem function $F$ is of class $\mathcal{C}^{1}$.
\end{remark}

Let $D\subset \mathbb{C}$ be an open set. Given a $\mathcal{C}^{1}$ stem function $F=F_1+\sqrt{-1}F_2:D\rightarrow \mathbb{H}_\mathbb{C}$, the two functions 
\begin{equation*}
 \frac{\partial F}{\partial z},\frac{\partial F}{\partial\bar{z}}:D\rightarrow\mathbb{H}_\mathbb{C},
\end{equation*}
are stem functions too. Explicitly:
\begin{equation*}
 \frac{\partial F}{\partial z}=\frac{1}{2}\left(\frac{\partial F}{\partial \alpha}-\sqrt{-1}\frac{\partial F}{\partial \beta}\right)=
 \frac{1}{2}\left(\frac{\partial F_1}{\partial\alpha}+\frac{\partial F_2}{\partial \beta}-\sqrt{-1}\left(\frac{\partial F_1}{\partial\beta}-\frac{\partial F_2}{\partial\alpha}\right)\right),
\end{equation*}
and
\begin{equation*}
 \frac{\partial F}{\partial \bar{z}}=\frac{1}{2}\left(\frac{\partial F}{\partial \alpha}+\sqrt{-1}\frac{\partial F}{\partial \beta}\right)=
 \frac{1}{2}\left(\frac{\partial F_1}{\partial\alpha}-\frac{\partial F_2}{\partial \beta}+\sqrt{-1}\left(\frac{\partial F_1}{\partial\beta}+\frac{\partial F_2}{\partial\alpha}\right)\right).
\end{equation*}
The previous stem functions induce the continuous slice derivatives:
\begin{equation*}
 \frac{\partial f}{\partial x}:=\mathcal{I}\left(\frac{\partial F}{\partial z}\right),\quad 
 \frac{\partial f}{\partial x^c}:=\mathcal{I}\left(\frac{\partial F}{\partial \overline{z}}\right).
\end{equation*}

While the spherical derivative controls the behavior of a slice function $f$ along the 
``spherical'' directions determined by $\mathbb{S}$, the slice derivatives 
$\partial/\partial x$ and $\partial /\partial x^{c}$, give information about the behavior along the
remaining directions (i.e.: along the (semi)slices). 

 If $f=\mathcal{I}(F):\Omega_D\rightarrow \mathbb{H}$, then we denote the 
restrictions over a  slice or a  semi-slice, as
\begin{equation*}
 f_J:=f|_{D_J}:D_J\rightarrow \mathbb{H},\quad f_J^+:=f|_{D_J^+}:D_J^+\rightarrow \mathbb{H},
\end{equation*}
respectively.

The following is a rewriting of a lemma contained in \cite{ghiloniperotti3}.
\begin{lemma}[\cite{ghiloniperotti3}, Lemma 2.1]\label{lemma2.1}
Let $f\in\mathcal{S}^{1}(\Omega_{D})$ and let $J\in\mathbb{S}$. Then, for each 
$x=\alpha+J\beta\in\Omega_D$, it holds:
\begin{equation*}
\frac{\partial f}{\partial x}(\alpha+J\beta)=\frac{\partial f_{J}}{\partial z_J}(\alpha+J\beta)\quad
\mbox{and}\quad\frac{\partial f}{\partial x^{c}}(\alpha+J\beta)=\frac{\partial f_{J}}{\partial \bar z_J}(\alpha+J\beta),
\end{equation*}
where $\partial/\partial z_J:=(1/2)(\partial/\partial \alpha -J\cdot\partial/\partial \beta)$
and $\partial/\partial \bar z_J:=(1/2)(\partial/\partial \alpha +J\cdot\partial/\partial \beta)$.
Furthermore, if $f\in\mathcal{S}^{\infty}(\Omega_{D})$ and $n\in\mathbb{N}$, then
\begin{equation*}
\frac{\partial^{n} f}{\partial x^{n}}(\alpha+J\beta)=\frac{\partial^{n} f_{J}}{\partial z^{n}_J}(\alpha+J\beta)
\end{equation*}
\end{lemma}
So, the slice derivatives at a certain point $x=\alpha+J\beta$ of a slice function $f$ 
can be computed by restricting the function to the proper semi-slice (in this case to 
$\mathbb{C}_J$), and then deriving with respect to $\partial/\partial z$ or 
$\partial/\partial\bar{z}$.

Now, left multiplication by $\sqrt{-1}$ defines a complex structure on $\mathbb{H}_{\mathbb{C}}$ and, with respect to this structure, a $C^{1}$ stem function
\begin{equation*}
F=F_{1}+\sqrt{-1}F_{2}:D\rightarrow\mathbb{H}_{\mathbb{C}}
\end{equation*}
is holomorphic if and only if, 
\begin{equation*}
\frac{\partial F}{\partial \overline{z}}\equiv 0.
\end{equation*}
We are now in position to define slice regular functions.

\begin{definition}[\cite{ghiloniperotti}]
 A slice function $f\in\mathcal{S}^1(\Omega_D)$ is called \textit{slice regular} if the following equation holds:
 \begin{equation*}
  \frac{\partial f}{\partial x^c}(\alpha+J\beta)=0,\quad \forall \alpha+J\beta\in\Omega_D.
 \end{equation*}
We denote by $\mathcal{SR}(\Omega_D)$ the real vector space of all slice regular functions on $\Omega_D$.
\end{definition}
\begin{remark}
Originally, slice regular functions were defined as functions 
$f:\Omega_{D}\subseteq\mathbb{H}\rightarrow\mathbb{H}$ such that,
for any $I\in\mathbb{S}$, the restriction $f_I$ has continuous partial derivatives and
$\partial f_{I}/\partial \bar {z}$
vanishes identically (cf. \cite{genstostru}, Definition 1.1). Anyway, if this definition implies \textit{sliceness} when 
$D\cap\mathbb{R}\neq\emptyset$, this is no more true in the general case. 
Furthermore, in \cite{ghiloniperotti2} it is shown that the class of quaternionic
functions which are holomorphic if restricted to any complex line $\mathbb{C}_{I}$
and are not \textit{slice} is too big and hence not very manageable.

\end{remark}

\textit{A slice regular function is, then, a slice function induced by a holomorphic stem function.} The next theorem gives a characterization of 
slice regular functions.

\begin{proposition}[\cite{ghiloniperotti}, Proposition 8 and Remark 6]
Let $f=\mathcal{I}(F)\in\mathcal{S}^{1}(\Omega_{D})$, then the following facts are equivalents:
\begin{itemize}
 \item $f\in\mathcal{SR}(\Omega_D)$;
 \item the restriction $f_{J}$ is holomorphic for every $J\in\mathbb{S}$ with respect to the complex structures on $D_{J}$ and $\mathbb{H}$ defined by left multiplication by $J$;
 \item two restrictions $f_{J}^+$, $f_{K}^+$ ($J\neq K$) are holomorphic on $D_{J}^{+}$ and $D_{K}^{+}$ respectively (the possibility $K=-J$ is not excluded).
\end{itemize}
\end{proposition}

Lemma 
\ref{lemma2.1}, implies that, if the set $D$ has nonempty
intersection with the real line, then $f$ is slice regular on $\Omega_D$ if and only if it is Cullen regular in the sense introduced by Gentili and Struppa in 
\cite{gentilistruppa1,gentilistruppa}. 

We recall that any slice regular function restricted to a slice admits a splitting into two complex holomorphic function
as the following lemma claims. A proof of this result can be found in \cite{colgensabstru} or in \cite{ghiloniperotti3}. 
In \cite{colgensabstru} the result is proven with the additional hypothesis that 
the domain of definition intersects the real axis. 
\begin{lemma}[\cite{colgensabstru}, Lemma 2.3]\label{splitting}
 Let $f\in\mathcal{SR}(\Omega_D)$ and $J\bot K$ two elements of $\mathbb{S}$. Then there exist two holomorphic functions
 $f_1,f_2:D_J\rightarrow \mathbb{C}_J$ such that 
 \begin{equation*}
  f_J=f_1+f_2K.
 \end{equation*}
\end{lemma}

\subsection{Slice product and zeros of slice functions}

The pointwise product of two slice functions is not, in general, a slice 
function\footnote{Take, for instance, $a\in\mathbb{H}\setminus\mathbb{R}$, $f(x)=xa$ and $g(x)=x$. 
Then, the lack of commutativity implies that, $h(x)=f(x)g(x)=xax$ is not a slice function.} 
but, if one considers the function induced by the pointwise
product of the two stem functions, then the result is a slice function and also regularity is preserved. 
So, to be more precise, we give the following definition.
\begin{definition}[\cite{ghiloniperotti}]
 Let $f=\mathcal{I}(F)$, $g=\mathcal{I}(G)$ be two slice functions on $\Omega_D$. The \textit{slice product} of $f$ and $g$ is the slice function defined by
 \begin{equation*}
  f\cdot g:=\mathcal{I}(FG)=\mathcal{I}(F_1G_1-F_2G_2+\sqrt{-1}(F_1G_2+F_2G_1)).
 \end{equation*}
\end{definition}

\begin{remark}
 In the previous definition, if the components of the first stem function $F=F_1+\sqrt{-1}F_2$ are real valued, then $(f\cdot g)(x)=f(x)g(x)$ for each $x\in\Omega_D$.
\end{remark}

\begin{definition}[\cite{ghiloniperotti}]
 A slice function $f=\mathcal{I}(F)$ is called \textit{real}, if the two components $F_1$ and $F_2$ are real-valued.
\end{definition}
The next proposition says that this notion of product is the good one,
meaning that it preserves regularity.
\begin{proposition}[\cite{ghiloniperotti}, Proposition 11]
 If $f,g\in\mathcal{SR}(\Omega_D)$, then $f\cdot g\in\mathcal{SR}(\Omega_D)$.
\end{proposition}
In \cite{ghiloniperotti} it is also pointed out that the regular product introduced in \cite{colgensabstru,gentilistoppato2}
is generalized by this one if the domain $\Omega_D$ does not have real points.
In the next proposition we explicit the slice product as the pointwise product
with the proper evaluations. This proposition was proved for regular functions defined on domains that intersect the real 
axis in \cite{colgensabstru,gensto,gentilistoppato2} and in \cite{altavilla} in this general
setting.
\begin{proposition}[\cite{altavilla}, Proposition 4.8]\label{prodcomp}
Let $f,g\in\mathcal{SR}(\Omega_{D})$ then, for any $x\in\Omega_{D}\setminus V(f)$
\begin{equation*}
(f\cdot g)(x)=f(x)g(f(x)^{-1}xf(x)).
\end{equation*}
\end{proposition}

Some notion about the zero set of a slice regular function will be useful in the next section. 
We will quote, then, the main results known in the literature.

If $F$ is a stem function, then $F^c$ is a stem function as well. We will denote by $f^c$ the slice function induced by $F^c$. The next definition given in
\cite{ghiloniperotti} generalizes the one given in \cite{gentilistoppato2} for power series. 
\begin{definition}[\cite{ghiloniperotti}]
 Let $f$ be a slice function over $\Omega_D$. Then we define the \textit{normal function} of $f$ (or \textit{symmetrization} of $f$) as the slice function 
 $N(f):=f\cdot f^c\in\mathcal{S}(\Omega_D)$.
\end{definition}
\begin{remark}
Let $f$ be a slice function. The following facts are contained in \cite{ghiloniperotti}, Section 6.
 \begin{itemize}
  \item If $f$ is a slice regular function, then also $f^c$ and $N(f)$ are slice regular functions.
  \item The following equation holds true:
  \begin{equation*}
   (f\cdot g)^c=g^c\cdot f^c,\quad \mbox{and so}\quad N(f)=N(f)^c.
  \end{equation*}
  Moreover, $N(f^c)= N(f)$.
  \item The next equality holds true:
  \begin{equation*}
   N(f\cdot g)=N(f)N(g).
  \end{equation*}
 \end{itemize}
\end{remark}

Let now $V(f)$ be the zero set of $f:\Omega_D\rightarrow\mathbb{H}$:
\begin{equation*}
 V(f):=\{x\in\Omega_D\,|\,f(x)=0\}.
\end{equation*}
The next proposition is a consequence of the ``affine behavior'' of slice regular
functions with respect to imaginary units.

\begin{proposition}[\cite{ghiloniperotti}, Proposition 16]
 Let $f\in\mathcal{S}(\Omega_D)$, then, for any $x\in\Omega_D\setminus\mathbb{R}$, the restriction of $f$ to 
 $\mathbb{S}_x$ is injective or constant.
\end{proposition}

A structure result for $V(f)$ is showed in the next theorem.
\begin{theorem}[\cite{ghiloniperotti}, Theorem 17]
 Let $f\in\mathcal{S}(\Omega_D)$ and let $x=\alpha+J\beta\in\Omega_D$. Then one of the following mutually exclusive statements holds:
 \begin{enumerate}
  \item $\mathbb{S}_x\cap V(f)=\emptyset$.
  \item $\mathbb{S}_x\subset V(f)$. In this case $x$ is called a real or spherical zero of $f$ if, respectively, $x\in\mathbb{R}$ or $x\notin \mathbb{R}$.
  \item $\mathbb{S}_x\cap V(f)=\{y\}$, with $y\notin \mathbb{R}$. In this case $x$ is called an $\mathbb{S}$-isolated non-real zero of $f$. 
 \end{enumerate}
\end{theorem}

\begin{proposition}[\cite{ghiloniperotti}, Proposition 25]\label{zerosprod}
Let $f,g\in\mathcal{S}(\Omega_{D})$. Then $V(f)\subset V(f\cdot g)$. Moreover it holds:
\begin{equation*}
\bigcup_{x\in V(f\cdot g)} \mathbb{S}_{x}=\bigcup_{x\in V(f)\cup V(g)}\mathbb{S}_{x}.
\end{equation*}
\end{proposition}


\begin{corollary}[\cite{ghiloniperotti}, Corollary 19]\label{corzerN}
 If $f$ is a real slice function then $f$ does not have $\mathbb{S}$-isolated non-real zeros. 
 Moreover, for any slice function $f\in\mathcal{S}(\Omega_D)$, if $x\in\Omega_D\cap V(f)$, then $\mathbb{S}_x\subset V(N(f))$. 
 Therefore, for any slice function $f$, it holds:
 \begin{equation*}
  V(N(f))=\bigcup_{x\in V(f)}\mathbb{S}_x.
 \end{equation*}
\end{corollary}

In the next theorem we add the hypothesis of regularity.
\begin{theorem}[\cite{ghiloniperotti}, Theorem 20]\label{preidentity}
 Let $\Omega_D$ be a connected circular domain and let $f$ be a slice regular function such that $N(f)$ does not vanish identically, then $V(f)\cap D_J$ is closed and 
 discrete in $D_J$ for every $J\in\mathbb{S}$.
\end{theorem}
In particular in \cite{altavilla,gentilistruppa,stoppatopoles} it is stated an Identity Principle.

\begin{theorem}[Identity principle]\label{identity}
Let $\Omega_D\subset \mathbb{H}$ be a connected domain and let $f:\Omega_{D}\rightarrow \mathbb{H}$ be a slice regular function.
\begin{itemize}
\item  \textbf{(\cite{stoppatopoles}, Proposition 3.3)} If $\Omega_D\cap\mathbb{R}\neq \emptyset$ and if there exists $I\in\mathbb{S}$ such that $D_{I}\cap V(f)$ has an accumulation point, 
then $f\equiv 0$ on $\Omega_{D}$. 
\item \textbf{(\cite{altavilla}, Theorem 3.6)} If $\Omega_D\cap\mathbb{R}= \emptyset$ and if there exist $K\neq J\in\mathbb{S}$ such that both 
$D_{K}^{+}\cap V(f)$ and $D_{J}^{+}\cap V(f)$ contain accumulation points, then $f\equiv 0$ on $\Omega_{D}$.
\end{itemize}
\end{theorem}
\begin{corollary}\label{identsp}
 Let $\Omega_D\subset \mathbb{H}$ be a connected domain and let $f:\Omega_{D}\rightarrow \mathbb{H}$ be a slice regular function.
If there exist a sphere $\mathbb{S}_x\subset\Omega_D\cap V(f)$ and a sequence $\{x_n\}_{n\in\mathbb{N}}\subset \Omega_D\setminus \mathbb{S}_x$
such that $x_n$ converges to $x$ and $\mathbb{S}_{x_n}\subset V(f)$ for any $n$, then $f$ vanishes identically.
\end{corollary}

\begin{remark}\label{identrk}
 In the proof of \cite{altavilla}, Theorem 3.6, it is actually proved that, if $f\in\mathcal{SR}(\Omega_D)$ and
 $V(f)\cap D_J^+$ contains an accumulation point, for some $J\in\mathbb{S}$, then the restriction $f_J^+$
 is identically zero.
\end{remark}

The distinction between the two cases in the previous theorem is underlined by the next example.
\begin{example}\label{exe1}
The slice regular function defined on $\mathbb{H}\setminus \mathbb{R}$ by
\begin{equation*}
f(x)=1-Ii,\quad x=\alpha +\beta I \in\mathbb{C}_{I}^{+}
\end{equation*}
is induced  by a locally constant stem function and its zero set $V(f)$ is the half plane $\mathbb{C}_{-i}^{+}\setminus\mathbb{R}$.
The function can be obtained by the representation formula in Theorem \ref{representationtheorem} by choosing the constant values $2$ on 
$\mathbb{C}_{i}^{+}\setminus\mathbb{R}$ and $0$ on $\mathbb{C}_{-i}^{+}\setminus \mathbb{R}$.
\end{example}

The notion of slice constant function was introduced in \cite{altavilla} to isolate the class of functions for which the previous example is a representative.
\begin{definition}[\cite{altavilla}]
 Let $f=\mathcal{I}(F)\in\mathcal{S}(\Omega_{D})$. $f$ is called \textit{slice constant} if the stem function $F$ is locally constant.
\end{definition}

\begin{proposition}[\cite{altavilla}, Theorem 3.4]
Let $f\in\mathcal{S}(\Omega_{D})$ be a slice constant function, then $f$ is slice regular.
Moreover $f$ is slice constant if and only if
\begin{equation*}
\frac{\partial f}{\partial x}=\mathcal{I}\left(\frac{\partial F}{\partial z}\right)\equiv 0.
\end{equation*}
\end{proposition}

The next definition is needed for defining the multiplicity of a slice function at a point.
\begin{definition}[\cite{ghiloniperotti}]
The characteristic polynomial of $y$ is the slice regular function $\Delta_y(x):\mathbb{H}\rightarrow\mathbb{H}$ defined by:
 \begin{equation*}
  \Delta_y(x):=N(x-y)=(x-y)\cdot (x-y^c)=x^2-x(y+y^c)+yy^c.
 \end{equation*}
\end{definition}
\begin{remark}
 The following facts about the characteristic polynomial are quite obvious. We refer the reader to \cite{ghiloniperotti}, Section 7.2.
 \begin{itemize}
  \item $\Delta_y$ is a real slice function.
  \item Two characteristic polynomials $\Delta_y$, $\Delta_y'$ coincide if and only if $\mathbb{S}_y=\mathbb{S}_{y'}$.
  \item $V(\Delta_y)=\mathbb{S}_y$.
 \end{itemize}
\end{remark}

It is showed in \cite{ghiloniperotti}, Corollary 23, that, if $f$ belongs to
$\mathcal{SR}(\Omega_D)$ and $x_0\in V(f)$, then $\Delta_{x_0}(x)$ divides $N(f)$.
Thanks to this fact we are able to give the following definition.
\begin{definition}[\cite{ghiloniperotti}]
 Let $f\in\mathcal{SR}(\Omega_D)$ such that $N(f)$ does not vanish identically.
 Given $n\in\mathbb{N}$ and $x_0\in V(f)$, we say that $x_0$ is a zero of $f$
 of \textit{total multiplicity} $n$, and we will denote it by $m_f(x_0)$,  
 if $\Delta_{x_0}^n\mid N(f)$ and $\Delta_{x_0}^{n+1}\nmid N(f)$. 
 If $m_f(x_0)=1$, then $x_0$ is called a \textit{simple zero} of $f$.
\end{definition}

The last definition, stated in \cite{ghiloniperotti}, is equivalent to the one of \textit{total multiplicity} stated in \cite{gentilistruppa2,genstostru}. 
The adjective ``total'' was introduced to underline the fact that this integer take
into accounts both spherical and isolated order of zero of a point. We will use this adjective in this paper 
to  distinguish the last notion of multiplicity to the one stated at the end of this section.

We recall now the definition of the degenerate set of a function.

\begin{definition}[\cite{genstostru}]
 Let $f\in\mathcal{S}(\Omega_D)$ and let $x=\alpha+I\beta\in\Omega_D$, $\beta>0$ be such that $\mathbb{S}_x=\alpha+\mathbb{S}\beta\subset \Omega_D$. 
 The 2-sphere $\mathbb{S}_x$ is said to be \textit{degenerate}
 for $f$ if the restriction $f|_{\mathbb{S}_x}$ is constant. The union $D_f$ of all 
 degenerate spheres for $f$ is called \textit{degenerate set} of $f$.
\end{definition}

Observe that the degenerate set of a slice function is a circular domain.
We will now state some properties of the degenerate set of a slice function.
First of all, the degenerate set of a slice function can be described as the zero set of the spherical
derivative as stated in the following proposition.

\begin{proposition}[\cite{altavilla}, Proposition 4.11]\label{degenerateprop}
 Let $f$ be a slice function over $\Omega_D$, then we have the following equality:
 \begin{equation*}
  D_f=V(\partial_s f).
 \end{equation*}
 Moreover $D_f$ is closed in $\Omega_{D}\setminus\mathbb{R}$.
\end{proposition}

\begin{proof}
 The proof of the statement is trivial thanks to Remark \ref{spherical}. 
 \end{proof}

As usual, adding the regularity property implies several additional results as
the following one.

\begin{proposition}[\cite{altavilla}, Proposition 4.12]\label{dfisempty}
If $f\in\mathcal{SR}(\Omega_{D})$ is non-constant, then, the interior of $D_{f}$ is empty.
\end{proposition}

In the next part of this subsection we say something more about the zero set of a slice regular function.
These results were exposed for the first time in the Ph.D. thesis \cite{altavillaPhD} of the author  and 
deal with the possibility, for a slice regular function $f:\Omega_D\rightarrow\mathbb{H}$, with $\Omega_D\cap\mathbb{R}=\emptyset$, 
to admit
surfaces $S_f\subset\Omega_D$ on which the function is constant.

\begin{theorem}\label{thmsurf}
Let $\Omega_{D}$ be a connected circular domain such that $\Omega_{D}\cap\mathbb{R}=\emptyset$. Let $f\in\mathcal{SR}(\Omega_{D})$ 
be a non-constant function. If  $x_{0}\in V(f)$ is not isolated in $V(f)$, then
there exists a real smooth surface $\mathcal{S}\subset\Omega_{D}$ such that $x_{0}\in \mathcal{S}\subset V(f)$.
\end{theorem}
\begin{proof}
Let $f=\mathcal{I}(F_1+\sqrt{-1}F_2)$ be a slice regular function, $x=\alpha+I\beta\in \Omega_D$ and $z=\alpha+i\beta\in D^+$.
If $x$ is an accumulation point in $V(f)\cap \mathbb{S}_x$ then, thanks to Theorem \ref{representationtheorem}, it is clear that the whole sphere
$\mathbb{S}_x$ is contained in the zero locus of $f$. Analogously, if $x$ is an accumulation point for
$V(f)\cap D_I^+$, then, thanks to Theorem \ref{identity} and Remark \ref{identrk}, $D_I^+\subset V(f)$. 
Since both $\mathbb{S}_x$ and $D_I^+$ are smooth surfaces for any $x$ and $I$, the theorem is proved in these two cases.

Let us consider then the case in which $x$ is a generic accumulation point that doesn't accumulate
in any sphere or in any semi-slice.
The point $x$ belongs to $V(f)$ if and only if $F_1(z)+IF_2(z)=0$. Since $x$ doesn't accumulate in any sphere that intersects $V(f)$,
then $F_2(z)\neq 0$.
Therefore, the zero locus of $f$ is equal to 
\begin{equation}\label{exeqn}
 V(f)=D_f^0\sqcup \{x\in D^+\times \mathbb{S}\,|\, x=(z,-F_1(z)F_2(z)^{-1})\in D^+\times\mathbb{S}\},
\end{equation}
where $D_f^0\subset D_f$, denotes the subset of the degenerate set of $f$, such that $f|_{D_f^0}\equiv 0$; i.e.: $D_f^0$
is the set of spherical zeros of $f$.

Since $x$ is an accumulation point in $V(f)\setminus D_f^0$, then there exist a sequence $\{x_n\}_{n\in\mathbb{N}}\subset V(f)\setminus D_f^0$
converging to $x$. Now, thanks to Corollary \ref{corzerN}, $N(f)$ vanishes on each $\mathbb{S}_{x_n}$ and since these spheres accumulate
to $\mathbb{S}_x$, then, thanks to Corollary \ref{identsp}, $N(f)\equiv 0$.
So, for any $z\in D^+$, there exists $I_z\in\mathbb{S}$, such that $(z,I_z)\in (D^+\times \mathbb{S})\cap V(f)$.
Now, the condition $N(f)\equiv 0$,
translates in the following system
\begin{equation}\label{surfzero}
\begin{cases}
g(F_{1}(z),F_{1}(z))-g(F_{2}(z),F_{2}(z))=0\\
g(F_{1}(z),F_{2}(z))=0.
\end{cases},\quad \forall z\in D^+,
\end{equation}
where $g$ is the standard euclidean product.
System \ref{surfzero} implies that, for any $z\in D^+$, $||-F_1(z)F_2(z)^{-1}||=1$ and the real part $Re(F_1(z)F_2(z)^{-1})=0$ and so $F_1(z)F_2(z)^{-1}\in\mathbb{S}$.
Finally,  the set 
\begin{equation*}
 \widetilde{V(f)}=\{x\in D^+\times \mathbb{S}\,|\, x=(z,-F_1(z)F_2(z)^{-1})\}
\end{equation*}
defines a real surface in $D^+\times\mathbb{S}$ that contains the accumulation point $x$.
To end the proof, since the function $f$ is slice regular, then, as expressed in \cite{ghiloniperotti2}, Remark 1.6, $f$ is
real analytic and so its inducing stem function has the same regularity. Therefore, the surface 
previously defined is smooth and the proof is concluded.
 \end{proof}

The zero locus of a non-constant slice regular function $f$, 
contains, then, isolated points, null-spheres and generic surfaces (possibly semi-slices), not contained in the degenerate set.
\begin{remark}
 The structure of the previous proof, even if it is quite elementary and naive,
 highlight the exact equations of the generic zero-surface contained in the 
 circular domain of a slice regular function (see equation \ref{exeqn}).
\end{remark}

Given a circular set $\Omega_D$ such that $\Omega_D\cap\mathbb{R}=\emptyset$, denote by
$p_1:\Omega_D\simeq D^+\cap \mathbb{S}\rightarrow D^+$ the projection over $D^+$: $p_1(\alpha+I\beta)=\alpha+i\beta$.

The following technical lemma will be useful in the last part of this paper.
\begin{lemma}\label{lemmasurf}
Let $\Omega_D$ be a connected circular open domain and $f\in\mathcal{SR}(\Omega_D)$ a slice regular function.
If there exists $q\in\mathbb{H}$ 
 such that $h=f-q$ admits two different non-degenerate smooth surfaces $S_1, S_2$, in 
 its zero locus (i.e.: $S_1, S_2\subset V(h)\setminus D_h$), for which $p_1(S_1)\cap p_1(S_2)\subset D^+$ contains an open set, 
 then $f$ is constant and equal to $q$.
\end{lemma}

\begin{proof}
 Without loss of generality, we can suppose $q=0$. 
 Then, for any $z\in p_1(S_1)\cap p_1(S_2)\subset D^+$ there exist $I_1\neq I_2\in\mathbb{S}$ such that $f$ vanishes 
 both at $(z,I_1)\in S_1$ and $(z,I_2)\in S_2$ in $\Omega_D=D^+\times \mathbb{S}$.
 This imply that the spherical derivative of $f$ is identically zero in $p_1(S_1)\cap p_1(S_2)$ 
 and so $f$ is constant and equal to zero.
 \end{proof}
\begin{remark}
Given a non-constant slice regular function $f$, the condition $N(f)\equiv 0$ defines a surface in $\Omega_{D}$ that could coincide with 
a semi-slice $D_{I}^{+}$, for some $I\in\mathbb{S}$, or not.
We will see in the next pages (see Lemma \ref{emptyinterior}), that the set of surfaces in which a slice regular function, that is not slice-constant,
is constant is contained in a possibly bigger set that is closed and with empty interior.  
\end{remark}

\begin{example}\label{exe2}
Let $g:\mathbb{H}\rightarrow\mathbb{H}$ be the slice regular function defined by
$g(x)=x+j$ and let $f$ be the slice regular function defined in example \ref{exe1}.
Consider now the slice regular function $h:\mathbb{H}\setminus\mathbb{R}\rightarrow\mathbb{H}$
defined by $h:=g\cdot f$. Explicitly, this function is defined by
\begin{equation*}
h(\alpha+I\beta)=\alpha+\beta i+j +I(\beta-\alpha i +k),
\end{equation*}
where $\alpha+I\beta$ belongs to $\mathbb{H}\setminus\mathbb{R}$.
The spherical derivative of $h$ is equal to
\begin{equation*}
\partial_{s}h(\alpha+I\beta)=1-\frac{\alpha}{\beta}i+\frac{k}{\beta},
\end{equation*}
that is always nonzero. Then, the function $h$ is not constant in any sphere.
We want to look for the zero set of $h$ and then we have to impose
the following equation:
\begin{equation*}
h(\alpha+I\beta)=\alpha+\beta i+j +I(\beta-\alpha i +k)=0
\end{equation*}
We remember Proposition \ref{zerosprod} which says that the zero set of the product
$h=g\cdot f$ is composed by the union of the zero set of  $g$ with the zero set of $f$
``properly modified'' (this ``modification'' is given by the formula in Proposition \ref{prodcomp}).
We have then that $h(-j)=0$. Suppose, now, $x\neq -j$. Then $h(x)=0$ if
and only if
\begin{equation*}
I=\frac{-(\alpha^{2}+\beta^{2}-1)i-2\beta j+2\alpha k}{\alpha^{2}+\beta^{2}+1}.
\end{equation*}
But then, the surface $S_{h}:\mathbb{C}^{+}\rightarrow\mathbb{H}\setminus\mathbb{R}$
defined by,
\begin{equation*}
S_{h}(\alpha+i\beta)=\left(\alpha+i\beta,\frac{-(\alpha^{2}+\beta^{2}-1)i-2\beta j+2\alpha k}{\alpha^{2}+\beta^{2}+1} \right)\subset\mathbb{C}^{+}\times\mathbb{S}\simeq\mathbb{H}\setminus\mathbb{R},
\end{equation*}
is a ``non trivial'' surface (i.e.: not a sphere nor a semi-slice), on which the slice
regular function $h$ result to be constant and equal to zero. 
Observe that $-j$ is in the image of $S_{h}$, in fact, for $S_{h}(i)=(i,-j)$,
therefore $V(h)=S_h$.

However, this surface is not the only $2$-dimensional manifold contained
in the domain of $h$ on which the function is constant. The function $h$
is, indeed, constant and equal to $2j$ on the semi-slice $\mathbb{C}_{-i}^{+}$. 
This was suggest by the fact that the slice derivative of $h$
is equal to $\partial h/\partial x(\alpha+I\beta)=1-Ii=f(\alpha+I\beta)$.

Later we will see that these are the only surfaces on which this function is constant.
\end{example}

\begin{remark}
 If $\Omega_D\cap\mathbb{R}\neq\emptyset$, then the fact that a slice regular function
 $f:\Omega_D\rightarrow\mathbb{H}$ is such that $N(f)\equiv 0$, implies that $f\equiv 0$ 
 (see, e.g., \cite{ghiloniperotti}, Theorem 20).
 Then, if there exists a sphere $\mathbb{S}_x\subset\Omega_D$ such that
 the zeros of $f$ in $\Omega_D\setminus\mathbb{S}_x$ accumulates to a 
 point of $\mathbb{S}_x$, this implies that $f\equiv 0$ 
 (cf \cite{genstostru}, Corollary 3.14).
\end{remark}

\subsection{Valence of a holomorphic function}

The next introductory tool come from complex analysis. The main reference for the following is \cite{heins}, Chapter V.9.

\begin{definition}
 Given a holomorphic function $f:D\subset\mathbb{C}\rightarrow\mathbb{C}$ we 
 define the \textit{multiplicity} of $f$ at a point $x\in D$ as the number:
 \begin{equation*}
  n(x;f):=\inf\{k\in\mathbb{N}\setminus \{0\}\,|\,f^{(k)}(x)\neq 0\},
 \end{equation*}
$f^{(k)}(x)$ denoting the $k^{th}$ derivative of $f$ with respect to $z$ evaluated in $x$.
\end{definition}

\begin{definition}
 Given a holomorphic function $f$ defined over a region $D$ we define the 
 \textit{valence} of $f$ at $w\in\mathbb{C}\cup\{\infty\}$ as
 \begin{equation*}
  v_f(w):=\begin{cases}
           +\infty & \mbox{if the set } \{f(z)=w\} \mbox{ is infinite};\\
           \sum_{f(z)=w}n(z;f) & \mbox{otherwise}.
          \end{cases}
 \end{equation*}
 \end{definition}
If $f$ does not take the value $w$, then $v_f(w)$ is obviously equal to zero.
\begin{remark}\label{valencerk}
If $f$ is a holomorphic function on a region $D$ and is not constant, then for any $r>0$, such that $\overline{D(x;r)}\subset D$, 
the valence at $w$ of $f|_{D(x;r)}$ is constant on each component of 
$(\mathbb{C}\cup\{\infty\})\setminus f(\partial D(x;r))$, where $D(x;r)$ denote the disc centered in $x$ of radius $r$.
\end{remark}

\subsection{Slice and spherical differentials of a slice function}

This last preparatory section contains material and ideas that were introduced by the author in \cite{altavillaPhD}.

Let $x\in\mathbb{H}\simeq\mathbb{R}^{4}$, $x=(x_0,x_1,x_2,x_3)$ with $(x_1,x_2,x_3)\neq(0,0,0)$ (i.e.: $x\in\mathbb{H}\setminus\mathbb{R}$). 
When we talk about slice functions we implicitly use the following change of coordinates:
\begin{equation*}
 (x_0,x_1,x_2,x_3)\mapsto (\alpha,\beta,I),
\end{equation*}
where $\alpha\in\mathbb{R}$, $\beta>0$ and $I=I(\vartheta,\varphi)\in\mathbb{S}$ with the following equalities:
\begin{equation*}\left\{
 \begin{array}{rcl}
  \alpha & = & x_0\\
  \beta & = & \sqrt{x_1^2+x_2^2+x_3^2}\\
  \vartheta & = & \arccos(\frac{x_3}{\beta})\\
  \varphi & = & \arctan(\frac{x_2}{x_1}).
 \end{array}\right.
\end{equation*}

Let now $f:\Omega\subset\mathbb{R}^4\rightarrow \mathbb{R}^4$ be any differentiable function. Then, its differential in these new coordinates, can be written
in its domain, as follows
\begin{equation}\label{diff}
 df=\left(\frac{\partial f}{\partial \alpha}d\alpha+\frac{\partial f}{\partial \beta}d\beta\right)+
 \frac{1}{\beta}\left(\frac{\partial f}{\partial\vartheta}d\vartheta+\frac{1}{\sin\vartheta}\frac{\partial f}{\partial\varphi}d\varphi\right),   
\end{equation}
where :
\begin{equation*}\left\{
 \begin{array}{rcl}
  d\alpha & = & dx_0\\
  d\beta & = & \sin\vartheta\cos\varphi dx_1+\sin\vartheta\sin\varphi dx_2+\cos\vartheta dx_3\\
  d\vartheta & = & \cos\vartheta\cos\varphi dx_1+\cos\vartheta\sin\varphi dx_2 -\sin\vartheta dx_3\\
  d\varphi & = & -\sin\varphi dx_1+\cos\varphi dx_2.
 \end{array}\right.
\end{equation*}

We would like, however, to consider also $\beta<0$ (having in mind that a non-real quaternion $x$ can be 
written both as $\alpha+I\beta$ and $\alpha+(-I)(-\beta)$). But in this case we have to take care that $d\beta(-\beta,I)=d\beta(\beta,-I)=-d\beta(\beta,I)$.

The aim of this section is to study the first part of the right hand side of equation \ref{diff}, when $f$ is a $\mathcal{C}^1$ slice function.

We will start with the following general definition.
\begin{definition}\label{slcfrm}
 Let $f=\mathcal{I}(F)\in\mathcal{S}^1(\Omega_D)$. We define the \textit{slice differential} $d_{sl}f$ of $f$ as the following differential form:
 \begin{equation*}
 \begin{array}{rrcc}
     d_{sl}f: &  (\Omega_D\setminus \mathbb{R})  & \rightarrow & \mathbb{H}^*,\\
   & \alpha+I\beta &  \mapsto  & dF_1(\alpha+i\beta)+IdF_2(\alpha+i\beta).
 \end{array}
    \end{equation*}
\end{definition}
\begin{remark}
The one-form $\omega:\mathbb{H}\setminus\mathbb{R}\rightarrow\mathbb{H}^*$ defined as $\omega(\alpha+I\beta)=Id\beta$, 
represents the outer radial direction to the sphere $\mathbb{S}_x=\{\alpha+K\beta\,|\,K\in\mathbb{S}\}$.
Then $\omega(\alpha+I(-\beta))=\omega(\alpha+(-I)\beta)=-\omega(\alpha+I\beta)$.
We can translate this observation in the language of slice forms.
The function $h(x)=Im(x)$ is a slice function induced by $H(z)=\sqrt{-1}Im(z)$.
Then we have $d_{sl}h(\alpha+I\beta)=Id\beta(\alpha+i\beta)$ and, thanks to the previous considerations
$d_{sl}h(\alpha+(-I)(-\beta))=-Id\beta(\alpha-i\beta)=Id\beta(\alpha+i\beta)$.
Summarizing, we have that $d\beta(\bar z)=-d\beta(z)$.
The same doesn't hold for $d\alpha$ which is a constant one form over $\mathbb{H}$ and 
for this reason in the next computations
we will omit the variable (i.e.: $d\alpha=d\alpha(z)=d\alpha(\bar z)$).
\end{remark}
We can show now that the previous definition is well posed.
\begin{proposition}
Definition \ref{slcfrm} is well posed, i.e. if $D$ is symmetric with respect to the real axis, then
 \begin{equation*}
  d_{sl}f(\alpha+I\beta)=d_{sl}f(\alpha+(-I)(-\beta)),\quad \forall \alpha+I\beta\in\Omega_D\setminus\mathbb{R}
 \end{equation*}
\end{proposition}
\begin{proof}
Let $x=\alpha+J\beta\in\Omega_D\setminus\mathbb{R}$ and $z=\alpha+i\beta$, then,
\begin{multline*}
d_{sl}f(\alpha+(-I)(-\beta))    = \\
  =\left(\displaystyle\frac{\partial F_1(\bar z)-IF_2(\bar z)}{\partial\alpha}\right) d\alpha+\left(\displaystyle\frac{\partial F_1(\bar z)-IF_2(\bar z)}{\partial\beta}\right) d\beta(\bar z)=\\
     =    \displaystyle\frac{\partial F_1}{\partial \alpha}(\bar z)d\alpha+\displaystyle\frac{\partial F_1}{\partial \beta}(\bar z)(-1)d\beta(z)-I\left(\displaystyle\frac{\partial F_2}{\partial \alpha}(\bar z)d\alpha+\displaystyle\frac{\partial F_2}{\partial \beta}(\bar z)(-1)d\beta(z)\right)=\\
     =    \displaystyle\frac{\partial F_1}{\partial \alpha}( z)d\alpha+\displaystyle\frac{\partial F_1}{\partial \beta}( z)d\beta(z)-I\left(-\displaystyle\frac{\partial F_2}{\partial \alpha}( z)d\alpha-\displaystyle\frac{\partial F_2}{\partial \beta}( z)d\beta(z)\right)=\\
     =    \left(\displaystyle\frac{\partial F_1( z)+IF_2( z)}{\partial\alpha}\right) d\alpha+\left(\displaystyle\frac{\partial F_1( z)+IF_2( z)}{\partial\beta}\right) d\beta(z)=\\
     =    d_{sl}f(\alpha+I\beta),
\end{multline*}
where the third equality holds thanks to the even-odd character of the couple $(F_1,F_2)$.
 \end{proof}
To avoid ambiguity, in the following of this section we will consider always $\beta> 0$ and we will omit the argument of the one-form $d\beta$.
We can represent, then, the slice differential as follows.
\begin{proposition}
 Let $f=\mathcal{I}(F)\in\mathcal{S}^1(\Omega_D)$ with $D\subset\mathbb{C}^+$ (so that $\beta>0$). Then, on $\Omega_D\setminus\mathbb{R}$, the following equality holds true.
 \begin{equation*}
     d_{sl}f=\frac{\partial f}{\partial \alpha}d\alpha+\frac{\partial f}{\partial \beta}d\beta.
    \end{equation*}
\end{proposition}
\begin{proof}
 The thesis follows from the following computations. Let $x=\alpha+I\beta\in\Omega_D$ and $z=\alpha+i\beta$, then
 \begin{equation*}
  \begin{array}{rcl}
   d_{sl}f(x) & = & \left(\frac{\partial F_1}{\partial \alpha}(z)d\alpha +\frac{\partial F_1}{\partial \beta}(z)d\beta\right)
   +I\left(\frac{\partial F_2}{\partial \alpha}(z)d\alpha +\frac{\partial F_2}{\partial \beta}(z)d\beta\right)\\
   & = & \left(\frac{\partial F_1}{\partial \alpha}(z)d\alpha +I\frac{\partial F_2}{\partial \beta}(z)d\beta\right)
   +\left(\frac{\partial F_1}{\partial \beta}(z)d\beta +I\frac{\partial F_2}{\partial \beta}(z)d\beta\right)\\
   & = & \frac{\partial f}{\partial \alpha}(x)d\alpha+\frac{\partial f}{\partial\beta}(x)d\beta.
  \end{array}
 \end{equation*}
 \end{proof}

It is clear from the definition that, if we choose the usual coordinate system, where $x=\alpha+I\beta$ with $\beta>0$, 
then $d_{sl}x=d\alpha +Id\beta$ and $d_{sl}x^c=d\alpha-Id\beta$.
We can now state the following theorem.

\begin{theorem}
 Let $f\in\mathcal{S}^1(\Omega_D)$. Then the following equality holds:
 \begin{equation*}
  d_{sl}x\frac{\partial f}{\partial x}(x)+d_{sl}x^c\frac{\partial f}{\partial x^c}(x)=d_{sl}f(x),\quad\forall x\in\Omega_D\setminus\mathbb{R}.
 \end{equation*}
\end{theorem}

\begin{proof}
 The thesis is obtained after the following explicit computations:
 \begin{equation*}
  \begin{array}{rcl}
   d_{sl}x\displaystyle\frac{\partial f}{\partial x}+d_{sl}x^c\displaystyle\frac{\partial f}{\partial x^c} & = & \displaystyle\frac{1}{2}\left[(d\alpha+Id\beta)\left(\displaystyle\frac{\partial F_1}{\partial\alpha}+\displaystyle\frac{\partial F_2}{\partial \beta}-I\left(\displaystyle\frac{\partial F_1}{\partial\beta}-\displaystyle\frac{\partial F_2}{\partial\alpha}\right)\right)+\right.\\
   &  & + \left.(d\alpha-Id\beta)\left(\displaystyle\frac{\partial F_1}{\partial\alpha}-\displaystyle\frac{\partial F_2}{\partial \beta}+I\left(\displaystyle\frac{\partial F_1}{\partial\beta}+\displaystyle\frac{\partial F_2}{\partial\alpha}\right)\right)\right]\\
   & = & \displaystyle\frac{1}{2}\left[d\alpha \displaystyle\frac{\partial F_1}{\partial \alpha}+d\alpha\displaystyle\frac{\partial F_2}{\partial\beta}-Id\alpha\displaystyle\frac{\partial F_1}{\partial \beta}+Id\alpha\displaystyle\frac{\partial F_2}{\partial \alpha}\right.\\
   & & +Id\beta\displaystyle\frac{\partial F_1}{\partial\alpha}+Id\beta\displaystyle\frac{\partial F_2}{\partial\beta}+d\beta\displaystyle\frac{\partial F_1}{\partial \beta}-d\beta\displaystyle\frac{F_2}{\partial \alpha}+\\
   & & +d\alpha \displaystyle\frac{\partial F_1}{\partial \alpha}-d\alpha\displaystyle\frac{\partial F_2}{\partial\beta}+Id\alpha\displaystyle\frac{\partial F_1}{\partial \beta}+Id\alpha\displaystyle\frac{\partial F_2}{\partial \alpha}+\\
   & & \left.-Id\beta\displaystyle\frac{\partial F_1}{\partial\alpha}+Id\beta\displaystyle\frac{\partial F_2}{\partial\beta}+d\beta\displaystyle\frac{\partial F_1}{\partial \beta}+d\beta\displaystyle\frac{F_2}{\partial \alpha}\right]\\
   & = & d\alpha\displaystyle\frac{\partial F_1}{\partial \alpha}+Id\beta\displaystyle\frac{\partial F_2}{\partial\beta}+d\beta\displaystyle\frac{\partial F_1}{\partial \beta}+Id\alpha\displaystyle\frac{\partial F_2}{\partial\alpha}\\
   & = & \displaystyle\frac{\partial F_1}{\partial \alpha}d\alpha+\displaystyle\frac{\partial F_1}{\partial\beta}d\beta+I\left(\displaystyle\frac{\partial F_2}{\partial \alpha}d\alpha+\displaystyle\frac{\partial F_2}{\partial\beta}d\beta\right)\\
   & = & d_{sl}f.
  \end{array}
 \end{equation*}
 \end{proof}

We have then the obvious corollary:
\begin{corollary}
 Let $f\in\mathcal{SR}(\Omega_D)$. Then the following equality holds:
 \begin{equation*}
  d_{sl}x\frac{\partial f}{\partial x}(x)=d_{sl}f(x),\quad\forall x\in\Omega_D\setminus\mathbb{R}.
 \end{equation*}
\end{corollary}

Given $f\in\mathcal{S}^{1}(\Omega_{D})$ we have seen that it is possible to define its 
slice differential, considering, roughly speaking, the restriction of its real differential, outside of the real line, to each semi-slice. 
It is clear that this object does not exhaust the description of the real differential. What we are going to define here is the missing part.

\begin{definition}
Let $f\in\mathcal{S}^{1}(\Omega_{D})$. 
We define the \textit{spherical differential} of $f$ as the following differentiable form:
\begin{equation*}
\begin{array}{c}
d_{sp}f:\Omega_{D}\setminus\mathbb{R}\rightarrow \mathbb{H}^{*},\\
d_{sp}f(\alpha+J\beta):=d_{\mathbb{R}}f(\alpha+J\beta)-d_{sl}f(\alpha+J\beta),
\end{array}
\end{equation*}
where $d_{\mathbb{R}}f(\alpha+J\beta)$ denote the real differential of $f$.
\end{definition}

We will give now  a more explicit description of the spherical differential of a slice function.
Starting from equation \ref{diff}, we have that,
\begin{equation*}
 d_{sp}f=df_{\mathbb{R}}-d_{sl}f= \frac{1}{\beta}\left(\frac{\partial f}{\partial\vartheta}d\vartheta+\frac{1}{\sin\vartheta}\frac{\partial f}{\partial\varphi}d\varphi\right),   
\end{equation*}
but since, for every $\alpha+J\beta\in\Omega_D\setminus\mathbb{R}$, $f$ depends on $J=J(\vartheta,\varphi)$ in an affine way, then, 
\begin{equation*}
 d_{sp}f=\displaystyle\frac{1}{\beta}\left(\displaystyle\frac{\partial J}{\partial\theta} d\theta+\displaystyle\frac{1}{\sin\theta}\displaystyle\frac{\partial J}{\partial\varphi}d\varphi \right)F_{2}.
\end{equation*}
But if $g:\mathbb{H}\rightarrow\mathbb{H}$ is the identity function, ($g(\alpha+I\beta)=\alpha+I\beta$), then
\begin{equation*}
d_{\mathbb{R}}g|_{\mathbb{H}\setminus\mathbb{R}}=d\alpha+Id\beta+(dI)\beta=d_{sl}x+\displaystyle\frac{1}{\beta}\left(\displaystyle\frac{\partial I}{\partial\theta} d\theta+\displaystyle\frac{1}{\sin\theta}\displaystyle\frac{\partial I}{\partial\varphi}d\varphi \right)\beta = d_{sl}x+d_{sp}x,
\end{equation*}
and so,
\begin{equation*}
 d_{sp}f=d_{sp}x\partial_sf
 \end{equation*}

It seems then that, if $f\in\mathcal{SR}(\Omega_D)$, then its real differential satisfies the following equation:
\begin{equation}\label{reprediffeq}
df|_{\Omega_D\setminus\mathbb{R}}=d_{sl}x\displaystyle\frac{\partial f}{\partial x}+d_{sp}x\partial_sf,
\end{equation}
where, the position of the elements of the cotangent space is on the \textit{left}.

As the reader could object, the previous are only formal considerations but, in the next pages everything will be proved in the case of slice regular functions
(in particular see Corollary \ref{reprediff}).
We remember firstly the notion of spherical analyticity and its consequences.
\section{The real differential of a slice function}

In this section we will describe the real differential of a slice function. 
For this purpose, in addition to what we already discussed in the previous pages, we will remember some results and constructions due to 
Caterina Stoppato (see \cite{stoppato}). 

\subsection{Coefficients of the spherical expansion}

In \cite{ghiloniperotti3,stoppato}, the authors introduce, in slightly different contexts, 
a spherical series of the form:
\begin{equation*}
 f(x)=\sum_{n\in\mathbb{N}}\mathcal{S}_{y,n}(x)s_n,
\end{equation*}
that gave some interesting results. 
More precisely, for each $m\in\mathbb{N}$ we define, the slice regular polynomial functions:
\begin{equation*}
 \mathcal{S}_{y,2m}(x):=\Delta_y(x)^m,\quad\mathcal{S}_{y,2m+1}(x):=\Delta_y(x)^m(x-y).
\end{equation*}
Note that, since $\Delta_y$ is a real slice function, then $\Delta_y^{\cdot m}=\Delta_y^{m}$.
Series of type $\sum_{n\in\mathbb{N}}\mathcal{S}_{y,n}(x)s_n$ have convergence sets that are always
open with respect to the Euclidean topology. In particular, they are open with respect to the following 
 \textit{Cassini} pseudometric $u$.
If $x,y\in\mathbb{H}$, we set 
\begin{equation*}
 u(x,y):=\sqrt{\mid\mid \Delta_y(x)\mid\mid}.
\end{equation*}
The function $u$ turn out to be a pseudometric on $\mathbb{H}$, whose induced topology is strictly coarser than the Euclidean one.
A $u$-ball of radius $r$ centered in $y$ will be denoted
by $U(y,R):=\{x\in\mathbb{H}\,|\, u(x,y)<R\}$.
In \cite{stoppato,ghiloniperotti3} it
is showed that the sets of convergence of series $\sum_{n\in\mathbb{N}}\mathcal{S}_{y,n}(x)s_n$ are $u$-ball centered at $y$ (see fig. 8.1 in \cite{genstostru}) and it is proved a corresponding
Abel Theorem. Moreover in \cite{ghiloniperotti3}, formulas for computing the coefficients are given.
In this context, the following is the definition of analyticity.
For the following definition we refer to \cite{stoppato,ghiloniperotti3}.
\begin{definition}
 Given a function $f:\Omega\rightarrow\mathbb{H}$ defined on a non-empty open circular subset $\Omega$ in $\mathbb{H}$, we say that $f$ is \textit{$u$-analytic} or \textit{spherical analytic}, if,
 for all $y\in\Omega$, there exists a non-empty $u$-ball $U$ centered at $y$ and contained in $\Omega$, and a series $\sum_{n\in\mathbb{N}}\mathcal{S}_{y,n}(x)s_n$
 with coefficients in $\mathbb{H}$, which converges to $f(x)$ for each $x\in U\cap\Omega$.
\end{definition}

We have the following expected result.
\begin{theorem}\label{sphexpa}
  Let $\Omega_D$ be a connected circular set and $f:\Omega_D\rightarrow\mathbb{H}$ be any function. The following assertions hold.
 \begin{enumerate}
  \item \textbf{(\cite{stoppato}, Corollary 4.3)} If $D\cap\mathbb{R}\neq\emptyset$, then$f$ is a slice regular function if and only if $f$ is a spherical analytic function.
  \item \textbf{(\cite{ghiloniperotti3}, Theorem 5.8)} If $D\cap\mathbb{R}=\emptyset$, then $f$ is a slice regular function if and only if $f$ is a spherical analytic slice function.
 \end{enumerate}
\end{theorem}

Given a slice regular function $f\in\mathcal{SR}(\Omega_D)$, the methods described in
\cite{stoppato,ghiloniperotti3} to compute its spherical coefficients $\{s_n\}$
at a fixed point, allow a correct explanation and interpretation only for the first two
(see, e.g., \cite{ghiloniperotti3}, formula 30):
\begin{equation*}
 \begin{array}{rcl}
  s_1 & = & \displaystyle\frac{1}{2}Im(y)^{-1}(f(y)-f(y^c))=\partial_sf(y)\\
  & & \\
  s_2 & = & \displaystyle\frac{1}{2}Im(y)^{-2}(2Im(y)\displaystyle\frac{\partial f}{\partial x}(y)-f(y)+f(y^c)),
 \end{array}
\end{equation*}
and in particular 
\begin{equation*}
 s_1+2Im(y)s_2=\frac{\partial f}{\partial x}(y).
\end{equation*}

The following proposition, which has an independent interest, allows us to understand better the nature of $s_{2}$.
\begin{proposition}
Let $f\in\mathcal{SR}(\Omega_{D})$ be a slice regular function, then the following formula holds:
\begin{equation}\label{formula}
\frac{\partial f}{\partial x}(x)=2Im(x)\left(\frac{\partial}{\partial x}\partial_{s}f\right)(x)+\partial_{s}f(x),\quad\forall x=\alpha+J\beta\in\Omega_{D}.
\end{equation}
\end{proposition}last

\begin{proof}
Let $F=F_{1}+\sqrt{-1}F_{2}$ the inducing stem function of $f$ and let $x=\alpha+J\beta\in\Omega_{D}\setminus\mathbb{R}$ and $z=\alpha+i\beta$, then,
\begin{equation*}
\displaystyle\frac{\partial f}{\partial x}(x)  =  \displaystyle\frac{1}{2}\left(\displaystyle\frac{\partial F_{1}}{\partial\alpha}(z)+
J\displaystyle\frac{\partial F_{2}}{\partial\alpha}(z)-J\displaystyle\frac{\partial F_{1}}{\partial\beta}(z)+
\displaystyle\frac{\partial F_{2}}{\partial\beta}(z)\right)=\circledast.
\end{equation*}
Using the slice regularity we have,
\begin{equation*}
 \circledast  = \displaystyle\frac{\partial F_{2}}{\partial\beta}(z)+J\displaystyle\frac{\partial F_{2}}{\partial\alpha}(z)=2J\left[\frac{1}{2}\left(\displaystyle\frac{\partial F_{2}}{\partial \alpha}(z)-J\displaystyle\frac{\partial F_{2}}{\partial \beta}(z)\right)\right](x).
\end{equation*}
Now $F_{2}(z)=\beta G(z)$, with $G=\left(F_2(z)/\beta\right)$ the stem function that induces the spherical derivative, then the last equation is equal to
\begin{equation*}
\begin{array}{rcl}
& & 2J\left[\displaystyle\frac{1}{2}\left(\beta\displaystyle\frac{\partial G}{\partial\alpha}(z)-J\beta\displaystyle\frac{\partial G}{\partial \beta}(z)-JG(z)\right)\right]=\\
& = & G(z)+2J\beta\left(\displaystyle\frac{1}{2}\left(\displaystyle\frac{\partial G}{\partial\alpha}(z)-J\displaystyle\frac{\partial G}{\partial\beta}(z)\right)\right)\\
& = & \partial_s f(x)+2Im(x)\left(\frac{\partial}{\partial x}\partial_{s}f\right)(x),
\end{array}
\end{equation*}
where of course, in the last equality $\partial_s f$ and $\frac{\partial}{\partial x}\partial_{s}f$ are the slice functions induced by $G$ and 
$\frac{1}{2}(\frac{\partial G}{\partial\alpha}-J\frac{\partial G}{\partial\beta})$ respectively.

At this point we have proven the theorem in the case in which the point $x$ is not real. Now, if the function $f$ is defined also on the real line,
then, thanks to slice regularity we have, in particular,  that $f$ is of class $\mathcal{C}^\infty$. Therefore, recalling Remark \ref{spherical}, we have that
the spherical derivative and its slice derivative extends continuously to the real line and the proof of the theorem is concluded.
 \end{proof}
\begin{remark}
Since the previous theorem holds for any $x_{0}\in\Omega_{D}$, then , if $x_{0}\in\mathbb{R}$,
then we have that
$\frac{\partial f}{\partial x}(x_{0})=\partial_{s}f(x_{0})$.
\end{remark}

\begin{corollary}
Let $f\in\mathcal{SR}(\Omega_{D})$ be a slice regular function with spherical expansion 
$ f(x)=\sum_{n\in\mathbb{N}}\mathcal{S}_{y,n}(x)s_n$ centered in $x_{0}\in\Omega_{D}$ then,
\begin{equation*}
s_{2}=\frac{\partial}{\partial x}(\partial_{s}f)(x_{0}).
\end{equation*}
\end{corollary}

\subsection{Rank of the real differential of a slice regular function}

In \cite{gensalsto,stoppato}, the authors shown the following theorem.
\begin{theorem}[\cite{stoppato}, Theorem 6.1]
 Let $f\in\mathcal{SR}(\Omega_D)$ and $x=\alpha+J\beta\in\Omega_D$. For all $v\in\mathbb{H}$, $||v||=1$, it holds
 \begin{equation*}
  \lim_{t\rightarrow 0}\frac{f(x+tv)-f(x)}{t}=vs_1+(xv-vx^c)s_2,
 \end{equation*}
where $s_1$ and $s_2$ are the first two coefficients of the spherical expansion of $f$.
\end{theorem}

The previous theorem has an important corollary (that is stated implicitly in \cite{gensalsto},
Section 3), which justifies explicitly the formal considerations in the introduction of Section 4.
\begin{corollary}\label{reprediff}
 Let $f\in\mathcal{SR}(\Omega_D)$ and let $(df)_x$ denote the real differential of $f$ at $x=\alpha+J\beta\in\Omega_D$. If we identify $T_x\mathbb{H}$ with
 $\mathbb{H}=\mathbb{C}_J\oplus \mathbb{C}_J^\bot$, then for all $v_1\in\mathbb{C}_J$ and $v_2\in\mathbb{C}_J^\bot$,
 \begin{equation*}
  (df)_x(v_1+v_2)=v_1\frac{\partial f}{\partial x}(x)+v_2\partial_sf(x).
 \end{equation*}
\end{corollary}

We will not give a proof of the previous theorem (and corollary) since the one in 
\cite{stoppato} does not use the additional hypothesis of nonempty intersection between
the domain and the real axis. The only feature needed for the proof is, in fact, 
the existence, 
for every slice regular function, of a spherical expansion.
But, as we state in Theorem \ref{sphexpa}, this is true also if the domain of definition 
of $f$ does not intersects the real line (cf. \cite{ghiloniperotti3}, Theorem 1.8).

The last corollary explains the last part of the previous section. We underline, in fact, the analogy 
between the representation of the real differential of a slice regular function given in Corollary \ref{reprediff}
and the one given in equation \ref{reprediffeq}.

We now want to study the rank of a slice regular function. In \cite{gensalsto} the authors proved that an injective slice regular function
defined over a circular domain with real points, has invertible differential. The aim of the following pages is to extend this result to all
slice regular functions. 
Let's start with a general result.

\begin{proposition}[\cite{gensalsto}, Proposition 3.3]\label{rank}
 Let $f\in\mathcal{SR}(\Omega_D)$ and $x_0=\alpha+J\beta\in\Omega_D\setminus\mathbb{R}$.
 \begin{itemize}
  \item If $\partial_sf(x_0)=0$ then:
  \begin{itemize}
   \item $df_{x_0}$ has rank 2 if $\frac{\partial f}{\partial x}(x_0)\neq 0$;
   \item $df_{x_0}$ has rank 0 if $\frac{\partial f}{\partial x}(x_0)= 0$.
  \end{itemize}
  \item If $\partial_sf(x_0)\neq0$ then $df_{x_0}$ is not invertible at $x_0$ if and only if 
  \begin{equation}\label{singcoeff}
  \frac{\partial f}{\partial x}(x_0)(\partial_sf(x_0))^{-1}\in\mathbb{C}_J^\bot.
  \end{equation}
 \end{itemize}
Let now $\alpha\in\Omega_D\cap \mathbb{R}$. $df_{x_0}$ is invertible at $\alpha$ if and only if its rank is not 0 at $x_0=\alpha +J\beta$. This happens if and only if
$\partial_sf(x_0)\neq 0$.
\end{proposition}

The proof of the previous statement can be found (with the appropriate change of notation), in \cite{gensalsto} or in \cite{genstostru}.
However, in our version of the statement, the last part involving the quantity in equation \ref{singcoeff} is much more clear and
directly computable.

\begin{remark}
 As the previous theorem states, the rank of the real differential a slice regular function is always an even number.
 This fact was also pointed out in \cite{perotti3D}, Corollary 4. Moreover it says
  that Theorem \ref{thmsurf} is optimal, meaning that no three-dimensional submanifold
  can be contained in the zero locus of a slice regular function.
\end{remark}

\begin{definition}
 Let $f:\Omega\rightarrow\mathbb{H}$ any quaternionic function of quaternionic variable. We define the singular set of $f$ as
 \begin{equation*}
  N_f:=\{x\in\Omega\,|\,df\mbox{ is not invertible at }x\}.
 \end{equation*}
\end{definition}

\begin{remark}
 If a slice regular function $f\in\mathcal{SR}(\Omega_D)$ is constant on  
 a surface $S$, then $S\subset N_f$. This is obvious if $S$ is in the degenerate set, but if $S$
 is not a degenerate sphere then this is true as well.
 If $S$ is a semi-slice $D_{I}^{+}$ for some $I\in\mathbb{S}$, then the slice derivative
 of $f$ on that semi-slice is everywhere zero and so $S\subset N_{f}$.
 Suppose now that $S$ is not in the degenerate set nor a semi-slice and $f|_S\equiv 0$. Then $N(f)\equiv 0$ and this equality translates in the system 
 in equation \ref{surfzero}.
 Deriving the first equation of system \ref{surfzero} with respect to $\beta$ and the second with respect to $\alpha$ we obtain,
 for each $z\in D$,
 \begin{equation*}
 \begin{cases}
 g(\frac{\partial F_{1}}{\partial\beta}(z),F_{1}(z))-g(\frac{\partial F_{2}}{\partial\beta}(z),F_{2}(z))=0\\
  g(\frac{\partial F_{1}}{\partial\alpha}(z),F_{2}(z))+g(\frac{\partial F_{2}}{\partial\alpha}(z),F_{1}(z))=0.
 \end{cases}
 \end{equation*}
If now $x_{0}=\alpha_{0}+I_{0}\beta_{0}\in S$ and $z_{0}=\alpha_{0}+i\beta_{0}\in D$, then
$f(x_{0})=0$, and so, if $S$ is not degenerate, $F_{1}(z_{0})=-I_{0}F_{2}(z_{0})$.
Evaluating the previous system in $z_{0}$ we obtain
\begin{equation*}
 \begin{cases}
 g(\frac{\partial F_{1}}{\partial\beta}(z_{0}),-I_{0}F_{2}(z_{0}))-g(\frac{\partial F_{2}}{\partial\beta}(z_{0}),F_{2}(z_{0}))=0\\
  g(\frac{\partial F_{1}}{\partial\alpha}(z_{0}),F_{2}(z))+g(\frac{\partial F_{2}}{\partial\alpha}(z_{0}),-I_{0}F_{2}(z_{0}))=0,
 \end{cases}
 \end{equation*}
 and, using regularity and the fact that, for any $p,q,r\in\mathbb{H}$, $g(pq,r)=g(q,p^{c}r)$,
 we get,
 \begin{equation*}
  \begin{cases}
g(I_{0}(\frac{\partial F_{1}}{\partial\beta}(z_{0})+I_{0}\frac{\partial F_{2}}{\partial\beta}(z_{0})),F_{2}(z_{0}))=0\\
g(I_{0}(\frac{\partial F_{1}}{\partial\alpha}(z_{0})+I_{0}\frac{\partial F_{2}}{\partial\alpha}(z_{0})),F_{2}(z_{0}))=0,
 \end{cases}
 \end{equation*}
 and so
  \begin{equation*}
  \begin{cases}
\beta_{0}||\partial_{s}f(x_{0})||g(\frac{\partial f}{\partial x}(x_{0})(\partial_{s}f)(x_{0})^{-1},1)=0\\
\beta_{0}||\partial_{s}f(x_{0})||g(\frac{\partial f}{\partial x}(x_{0})(\partial_{s}f)(x_{0})^{-1},I_{0})=0,
 \end{cases}
 \end{equation*}
therefore, for any $x_{0}\in S$, we have that $x_{0}\in N_{f}$.
\end{remark}

\begin{example}\label{exe3}
We will compute now the singular set $N_{h}$ of the function 
$h:\mathbb{H}\setminus\mathbb{R}\rightarrow\mathbb{H}$ defined
in example \ref{exe2},
\begin{equation*}
h(\alpha+I\beta)=\alpha+\beta i+j+I(\beta-\alpha i+k),\quad \alpha+I\beta\in\mathbb{H}\setminus\mathbb{R}.
\end{equation*}
We have seen that $\partial_{s}h\neq 0$ and so, thanks to Proposition \ref{degenerateprop}, 
$D_{f}=\emptyset$. But then, as it is stated in Proposition \ref{rank}, a point 
$x\in\mathbb{H}\setminus\mathbb{R}$ belongs to $N_{h}$ if and only if
 \begin{equation*}
  \frac{\partial h}{\partial x}(x)(\partial_sh(x_0))^{-1}\in\mathbb{C}_I^\bot.
  \end{equation*}
 Based on the computations in example \ref{exe2}, the last can be written explicitly as,
 \begin{equation*}
  (1-Ii)\frac{(\beta+\alpha i-k)}{\beta||\partial_{s}h(\alpha+I\beta)||^{2}}\in\mathbb{C}_I^\bot.
  \end{equation*}
  After some computation, using the ``scalar-vector'' notation, we obtain that the previous 
  relation is satisfied if and only if
  \begin{equation*}\left\{
  \begin{array}{rcl}
  g(\frac{\partial f}{\partial x}(x)(\partial_{s}f)(x)^{-1},1) & = & 0\\
  g(\frac{\partial f}{\partial x}(x)(\partial_{s}f)(x)^{-1},I) & = & 0
  \end{array}\right.\quad\Leftrightarrow\quad
\left\{
  \begin{array}{rcl}
   \beta(1+g(I,i))+g(I,j) & = & 0\\
   \alpha(1+g(I,i))-g(I,k) & = & 0
  \end{array}\right.
  \end{equation*}
This condition is clearly verified for any $x\in\mathbb{C}_{-i}^{+}$ that is the semi-slice
on which $h$ is constant and equal to $2j$. Suppose then that $I\neq -i$
and write $I=Ai+Bj+Ck$. The previous system become
  \begin{equation*}
\left\{
  \begin{array}{rcl}
   \beta & = & -B(1+A)^{-1}\\
   \alpha & = & C(1+A)^{-1},
  \end{array}\right. 
  \end{equation*}
  and so, for any $I\in\mathbb{S}$ such that $-B(1+A)^{-1}>0$ there exists a point 
  $\alpha+i\beta\in\mathbb{C}^{+}$ such that $\alpha+I\beta\in N_{h}$. We want to show
  now that the set of quaternions that satisfies these requirements is contained in the 
  surface $S_{h}$ defined in example \ref{exe2}.
  But if $x=C(1+A)^{-1}-(Ai+Bj+Ck)B(1+A)^{-1}$, then,
  \begin{equation*}
 \begin{array}{rcl}
 h(x) & = & C(1+A)^{-1}-B(1+A)^{-1}i+\\
 & & +j(Ai+Bj+Ck)(-B(1+A)^{-1}-C(1+A)^{-1}i+k)\\
 & = & C(1+A)^{-1}+AC(1+A)^{-1}-C+\\
 &  & +(-B(1+A)^{-1}-AB(1+A)^{-1}+B)i+\\
 & & +(1-A-B^{2}(1+A)^{-1}-C^{2(1+A)^{-1}})j+\\
& & + (BC(1+A)^{-1}-BC(1+A)^{-1})k\\
 & = & 0,
 \end{array}
  \end{equation*}
and since the zero set of $h$ is exactly the surface $S_{h}$ in example \ref{exe2}, then
$N_{h}=\mathbb{C}_{-i}^{+}\cup S_{h}$.

Since, now, the set of surfaces on which $h$ is constant is contained in 
$N_{h}$, we obtain that $\mathbb{C}_{-i}^{+}$ and $S_{h}$ are the only
two surfaces, contained in $\mathbb{H}\setminus\mathbb{R}$ on which
$h$ is constant.
  
\end{example}

The following theorem will characterize the set $N_f$ of singular points of $f$.
In particular, the next theorem generalizes a well known concept in real and complex
analysis i.e.: the fact that if the differential of a function is singular in some
point $x_0$, then, the function can be expanded in a neighborhood of $x_0$ as
\begin{equation*}
 f(x)=f(x_0)+o((x-x_0)^2).
\end{equation*}

\begin{theorem}[\cite{gensalsto}, Proposition 3.6]\label{thmchar}
Let $f\in\mathcal{SR}(\Omega_D)$ and let $x_0=\alpha +\beta I\in\Omega_D$. Then $x_0\in N_f$
if and only if there exists a point $\widetilde{x_0}\in\mathbb{S}_{x_0}$ and a function $g\in\mathcal{SR}(\Omega_D)$
such that the following equation hold:
\begin{equation*}
f(x)=f(x_0)+(x-x_0)\cdot(x-\widetilde{x_0})\cdot g(x).
\end{equation*}
Equivalently, $x_0\in N_f$ if and only if the function $f-f(x_0)$ has total multiplicity $n\geq 2$ in 
$\mathbb{S}_{x_0}$.
\end{theorem}

The proof of the last theorem in \cite{gensalsto} does not
use the hypothesis $\Omega_D\cap\mathbb{R}\neq\emptyset$. However, since
our setting and our notation are quite distant from \cite{gensalsto},
we will rewrite the proof. 
Before proving the last theorem we recall from \cite{gensalsto} the following remark.
\begin{remark}
For all $x_0=\alpha+J\beta\in\mathbb{H}\setminus\mathbb{R}$, setting $\Psi(x):=(x-x_0)(x-x_0^c)^{-1}$ defines a stereographic projection of 
$\alpha+\mathbb{S}\beta$ onto the plane $\mathbb{C}_J^\bot$ from the point $x_0^c$.
Indeed, if we choose $K\in\mathbb{S}$ with $K\bot J$ then for all $x=\alpha+\beta L$ with $L=tJ+uK+vJK\in\mathbb{S}$ we have 
$\Psi(x)=(L-J)(L+J)^{-1}=\frac{u+vJ}{1+t}JK$ and $\mathbb{C}_J\cdot K=(\mathbb{R}+\mathbb{R}J)JK=\mathbb{C}_J^\bot$.
\end{remark}

We are now able to pass to the proof of the theorem.

\begin{proof}
 If $x_0\in\Omega_D\setminus \mathbb{R}$ then it belongs to $D_f$ if and only if, $f$ is constant on the sphere $\mathbb{S}_{x_0}$, i.e. 
 there exists a slice regular function $g:\Omega_D\rightarrow \mathbb{H}$ such that 
 \begin{equation*}
  f(x)-f(x_0)=\Delta_{x_0}(x)g(x).
 \end{equation*}
This happens if and only if the coefficient $s_1=\partial_s f(x_0)$ in the spherical expansion vanishes.

Let now pass to the case $x_0\in\Omega_D\setminus \mathbb{R}$, $x_0\notin D_f$. Thanks to Proposition \ref{rank}, $x_0\in N_f$ if and only if,
$1+2Im(x_0)s_2s_1^{-1}=p\in\mathbb{C}_J^\bot$. Thanks to the previous remark, $p\in\mathbb{C}_J^\bot$ if and only if
there exists $\widetilde{x_0}\in \mathbb{S}_{x_0}\setminus \{x_0^c\}$ such that
$p=(\widetilde{x_0}-x_0)(\widetilde{x_0}-x_0^c)^{-1}$. The last formula is equivalent to
\begin{equation*}
\begin{array}{rcl}
  2Im(x_0)s_2s_1^{-1} & = & (\widetilde{x_0}-x_0)(\widetilde{x_0}-x_0^c)^{-1}-(\widetilde{x_0}-x_0^c)(\widetilde{x_0}-x_0^c)^{-1}\\
  & = & (\widetilde{x_0}-x_0-\widetilde{x_0}+x_0^c)(\widetilde{x_0}-x_0^c)^{-1}\\
  & = & -2Im(x_0)(\widetilde{x_0}-x_0^c)^{-1},
\end{array}
\end{equation*}
that is $s_1=(x_0^c-\widetilde{x_0})s_2$. Writing then the first terms of the spherical expansion of $f$ around $x_0$ we have:
\begin{equation*}
 \begin{array}{rcl}
  f(x) & = & s_0 +(x-x_0)s_1+\Delta_{x_0}(x)s_2+\Delta_{x_0}(x)(x-x_0)\cdot h(x)\\
  & = & s_0 +(x-x_0)(x_0^c-\widetilde{x_0})s_2+\Delta_{x_0}(x)s_2+\Delta_{x_0}(x)\cdot (x-x_0)h(x)\\
  & = & s_0+(x-x_0)(x_0^c-\widetilde{x_0})s_2+\\
  & & +\Delta_{x_0}(x)s_2+(x-x_0)\cdot (x-x_0^c)\cdot (x-x_0)\cdot h(x)\\
  & = & s_0 +(x-x_0)\cdot [(x_0^c-\widetilde{x_0}+x-x_0^c)s_2+\Delta_{\widetilde{x_0}}(x)h(x)]\\
  & = & s_0 +(x-x_0)\cdot (x-\widetilde{x_0})\cdot [s_2+(x-\widetilde{x_0^c}h(x))]\\
  & = & f(x_0) +(x-x_0)\cdot(x-\widetilde{x_0})\cdot [s_2+(x-\widetilde{x_0^c}h(x))],
 \end{array}
\end{equation*}
for some slice regular function $h:\Omega_D\rightarrow \mathbb{H}$, where we used the following facts:
\begin{itemize}
 \item $(x-x_0)(x_0^c-\widetilde{x_0})=(x-x_0)\cdot(x_0^c-\widetilde{x_0})$ because the second factor is constant;
 \item $\Delta_{x_0}(x)(x-x_0)=\Delta_{x_0}(x)\cdot (x-x_0)$ because the first factor is a real sli\-ce function;
 \item $(x-x_0^c)\cdot (x-x_0)=\Delta_{x_0}(x)$;
 \item $\Delta_{x_0}(x)=\Delta_{\widetilde{x_0}}(x)$ because $\widetilde{x_0}\in\mathbb{S}_{x_0}$.
\end{itemize}

Finally, if $x_0\in\Omega_D\cap\mathbb{R}$ then $s_1=0$ if and only if
\begin{equation*}
 f(x)=f(x_0)+(x-x_0)^2\cdot l(x)=f(x_0)+(x-x_0)\cdot(x-x_0)\cdot l(x),
\end{equation*}
for some slice regular function $l:\Omega_D\rightarrow \mathbb{H}$.
 \end{proof}

For the main results we need, now, two lemmas.
\begin{lemma}\label{emptyinterior}
 Let $f:\Omega_D\rightarrow \mathbb{H}\in\mathcal{SR}(\Omega_D)$ be non slice-constant. Then its singular set $N_f$ is closed and with empty interior.
\end{lemma}

\begin{proof}
By Proposition \ref{dfisempty}, $D_f$ has empty interior.
Since $D_f=V(\partial_sf)$ then it is closed in $\Omega_D$.
 So, since $D_f\subset N_f$, then the thesis is that $N_f\setminus D_f$ is closed and has empty interior. 
 
 To show that $N_f\setminus D_f$ is closed it is sufficient to observe
 that, for any $y=\alpha+I\beta$ in this set
 \begin{equation*}
\frac{\partial f}{\partial x}(y)\partial_{s}f(y)^{-1}\in \mathbb{C}_{I}^{\bot}
\end{equation*}
and this is true if and only if 
\begin{equation*}
\frac{\partial f}{\partial x}(y)\partial_{s}f(y)^{-1}Im(y)=-Im(y)\frac{\partial f}{\partial x}(y)\partial_{s}f(y)^{-1}.
\end{equation*}
But then 
\begin{equation*}
N_f\setminus D_f=\{y\in\Omega_{D}\,|\,
\frac{\partial f}{\partial x}(y)\partial_{s}f(y)^{-1}Im(y)+Im(y)\frac{\partial f}{\partial x}(y)\partial_{s}f(y)^{-1}=0\},
\end{equation*}
and so it is the pre-image, via a continuous function, of a closed set.

 Let $x_{0}\in N_{f}\setminus D_{f}$ and by contradiction let $R>0$ be a real number 
 such that the open Euclidean ball $B=B(x_{0},R)$ centered in $x_{0}$ with radius 
 $R$ is
 fully contained in $N_{f}\setminus D_{f}$. For any $y\in B$ the spherical derivative 
 $\partial_{s}f(x_{0})\neq 0$ and, by Theorem \ref{thmchar}, there
 exists a slice regular function $h_{y}:\Omega_{D}\rightarrow \mathbb{H}$ such that
 $N(f-f(y))=\Delta_{y}(x)^{2}h(x)$, where $N(f-f(y))$ is the normal function of $f-f(y)$.
 Computing the slice derivative of $N(f-f(y))$ and evaluating in $x=y$ we obtain
 \begin{equation*}
 0=\left[\frac{\partial N(f-f(y))}{\partial x}\right]_{x=y}=\left[\frac{\partial f}{\partial x}\cdot (f-f(y))^{c}\right]_{x=y}.
 \end{equation*}
 There are two cases $1)$ $\frac{\partial f}{\partial x}(y)=0$ or $2)$ $\frac{\partial f}{\partial x}(y)\neq0$. Case $2)$ implies, using formula in Proposition \ref{prodcomp}, that 
 \begin{equation*}
 f\left(\frac{\partial f}{\partial x}(y)^{-1}y\frac{\partial f}{\partial x}(y)\right)=f(y).
 \end{equation*}
Case $1)$ can be divided into two sub-cases: $i)$ $y=\alpha+I\beta$ is an isolated zero for 
the slice derivative in $D_{I}^{+}$ or $ii)$ $\frac{\partial f}{\partial x}_{I}^{+}\equiv 0$.
If  $ii)$ holds true, then we change our point $y$ considering another point $\omega\in B$ lying on another different semi-slice. 
Then, $\omega$ can only be an isolated zero on its semi-slice for the slice derivative of $f$ (otherwise, thanks to Lemma \ref{lemmasurf}, 
$\partial f/\partial x$ would be constant and $f$ would, then, be slice-constant). 
The only possibility is, therefore, case $i)$. If we are in case $1),i)$ then we can find a positive real number $r$
such that the two dimensional disc $\Delta =\Delta_{I}(x_{0},r)$ is contained in $B\cap\mathbb{C}_{I}^{+}$ and, 
for any $x\in \Delta\setminus\{y\}$ we have $\frac{\partial f}{\partial x}(x)\neq 0$. For any $y'\in\Delta\setminus\{y\}$ we are in case $2)$ and, again, there are two
sub cases: $A)$ $\frac{\partial f}{\partial x}(y')^{-1}y'\frac{\partial f}{\partial x}(y')\neq y'$ or
$B)$ $\frac{\partial f}{\partial x}(y')^{-1}y'\frac{\partial f}{\partial x}(y')=y'$.
If there is a point that satisfies case $A)$, then $f$ would be equal to some quaternion 
$p$ both in $y'$ and in $\frac{\partial f}{\partial x}(y')^{-1}y'\frac{\partial f}{\partial x}(y')$ 
and this would implies, using the representation theorem, that $f|_{\mathbb{S}_{y}}\equiv p$ that is $\mathbb{S}_{y}\in D_{f}$.
So, the only possible case is, finally, $B)$. But if condition $B)$ holds true for any 
$y\in \Delta\setminus{y'}$, then 
\begin{equation*}
y\frac{\partial f}{\partial x}(y)=\frac{\partial f}{\partial x}(y)y,
\end{equation*}
and so, for any $y\in\Delta\setminus{y'}$, $\frac{\partial f}{\partial x}(y)$ belongs to $\mathbb{C}_{I}$, therefore, 
thanks to Theorem \ref{identity}, this is true for any point in
$D_{I}^{+}$.
We claim that this is not possible. In fact, if $\alpha+I\beta=y\in B$, then 
\begin{equation*}
\frac{\partial f}{\partial x}(y)\partial_{s}f(y)^{-1}\in \mathbb{C}_{I}^{\bot}
\end{equation*}
and, as before,  this is true if and only if 
\begin{equation}\label{eqsing}
\frac{\partial f}{\partial x}(y)\partial_{s}f(y)^{-1}Im(y)=-Im(y)\frac{\partial f}{\partial x}(y)\partial_{s}f(y)^{-1}.
\end{equation}
But $\frac{\partial f}{\partial x}(y)$ belongs to $\mathbb{C}_{I}$ then it commutes
with $Im(y)$ and so, from the previous equation \ref{eqsing}, we get,
\begin{equation*}
\partial_{s}f(y)^{-1}Im(y)=-Im(y)\partial_{s}f(y)^{-1} 
\end{equation*}
which means that $\partial_{s}f(y)\in \mathbb{C}_{I}^{\bot}$ for each $y\in D_{I}^{+}$. This implies
that there exists an imaginary unit $J\in\mathbb{S}$ orthogonal to $I$ and a function
$g:D_{I}^{+}\rightarrow \mathbb{R}$ such that, for any $y\in D_{I}^{+}$ it holds
$\partial_{s}f(y)=\frac{1}{\beta}g(y)J$. Since the spherical derivative is independent from the
imaginary unit $I$ then it is $g$ too. Since $f=\mathcal{I}(F_{1}+\sqrt{-1}F_{2})$ is a slice regular function, then
\begin{equation*}
\left(\frac{\partial f}{\partial x}\right)_{I}=\frac{\partial F_{2}}{\partial\beta}-I\frac{\partial F_{2}}{\partial\alpha}=\left(\frac{\partial g}{\partial\beta}-I\frac{\partial g}{\partial\alpha}\right)J
\end{equation*}
and this is not possible since, as we said, the slice derivative belongs to $\mathbb{C}_{I}$.
\end{proof}

\begin{lemma}
 Let $f=\mathcal{I}(F):\Omega_D\rightarrow \mathbb{H}$ be an injective slice function. Then for any $x=\alpha +J\beta\in\Omega_D\setminus\mathbb{R}$, $\partial_sf(x)\neq 0$.
\end{lemma}
\begin{proof}
 We know that $\partial_sf(x)=0$ if and only if $f$ is constant on the sphere $\mathbb{S}_x$ (see Remark \ref{spherical}).  But then if $f$ is injective
 then $\partial_sf(x)\neq 0$ for all $x\in\Omega_D\setminus\mathbb{R}$.
 \end{proof}

Now we have that every injective slice regular function has real differential with rank at least equal to 2. The next step is to prove that 
for every injective slice regular function $f$ the slice derivative $\frac{\partial f}{\partial x}$ is everywhere different from 0.
\begin{theorem}\label{thmslder}
 Let $f=\mathcal{I}(F):\Omega_D\rightarrow\mathbb{H}$ be an injective slice regular function. Then its slice derivative $\frac{\partial f}{\partial x}$ is always 
 different from zero.
\end{theorem}
\begin{proof}
 What we want to prove is that, for any $x_0=\alpha +J\beta\in\Omega_D$
 \begin{equation*}
  \frac{\partial f}{\partial x}(x_0)\neq 0.
 \end{equation*}
 First of all, thanks to Theorem \ref{preidentity} applied to the slice derivative of $f$, if $\partial f/\partial x$ is equal to zero in $y\in D_I^+\subset\Omega_D$,
 for some $I\in\mathbb{S}$, then, either $D_I^+\subset V(\partial f/\partial x)$, or $y$ is isolated in $D_I\cap V(\partial f/\partial x)$.
 But if $D_I^+\subset V(\partial f/\partial x)$, then $f_I^+$ is constant and since, by hypothesis, $f$ is injective, then this is not possible.
 Since $f$ is slice regular, then, thanks to Lemma \ref{splitting}, for any 
 $J\bot K\in\mathbb{S}$ there exist two holomorphic functions 
 $f_1,f_2:D_J\rightarrow \mathbb{C}_J$ such that $f_J=f_1+f_2K$.

 Thanks to Lemma \ref{lemma2.1} we have then that, 
 \begin{equation*}
 \frac{\partial f}{\partial x}(x_0)=\frac{\partial f_1}{\partial z_J}(x_{0})+\frac{\partial f_2}{\partial z_J}(x_{0})K,
\end{equation*}
where $\partial/\partial z_J=1/2(\partial/\partial\alpha-J\cdot\partial/\partial\beta)$
and so, since $f_1$ and $f_2K$ lives on independent subspaces of $\mathbb{H}$,
the thesis become that at least one of the two derivatives 
$\frac{\partial f_1}{\partial z_J}(x_{0})$, $\frac{\partial f_2}{\partial z_J}(x_{0})$
is different from zero.
Moreover, since $f$ is injective, then also $f_J$ is injective.
So, if one between $f_1$ and $f_2$ is constant, then the
other one must be injective, and so we will have an injective holomorphic 
function and the thesis will follow trivially.
Let's suppose then that both $f_1$ and $f_2$ are non-constant functions and fix 
the following notations:
\begin{equation*}
  n(x;f):=\inf\{k\in\mathbb{N}\setminus \{0\}\,|\,\frac{\partial^k f}{\partial x^k}(x)\neq 0\},
 \end{equation*}
 \begin{equation*}
  n_1(x;f):=\inf\{k\in\mathbb{N}\setminus \{0\}\,|\, f^{(k)}_1(x)\neq 0\},
 \end{equation*}
 \begin{equation*}
  n_2(x;f):=\inf\{k\in\mathbb{N}\setminus \{0\}\,|\, f^{(k)}_2(x) \neq 0\},
 \end{equation*}
 where $f^{(k)}_i$ denotes the $k$-th derivative of $f_i$ with respect to $\partial/\partial z_J$
 Using again  Lemma \ref{lemma2.1}, we have that, for every $x\in D_J$, 
 \begin{equation*}
n(x;f)=\min(n_1(x;f),n_2(x;f)).  
 \end{equation*}
 Moreover, since $f$ is non slice-constant then the null set 
 of its slice derivative restricted to the semi-slice $D_J$ is discrete.
 For this reason we can take two balls $B_1:=B_1(x_0;r_1)$, $B_2:=B_2(x_0;r_2)$, such that their closure 
 is contained in $D_{J}^{+}$, $f_i$ take the value $f_i(x_0)$ on $\overline{B_i}$ only 
 at $x_0$ and $f_i'(z)\neq 0$ for any $z\in B_i\setminus \{x_0\}$.
 Let now $B=B_1\cap B_2$, then, as it is pointed out
 in Remark \ref{valencerk}, the valence $v_{f_i}(f_i(z))$ of $f_i|_B$  is constant and equal 
 to $n_i(z;f)$ in the component of  $(\mathbb{C}_J\cup\{\infty\})\setminus f(\partial B)$ which contains $f_{i}(z)$. 
 Since $n(x;f)=\min(n_1(x;f),n_2(x;f))$ and $n(x;f)=1$ almost everywhere, then 
 $\exists y\in B$ and $j\in\{1,2\}$ such that $1=n(y;f)=n_j(y;f)$. 
 Then $n_j$ is constant and
 equal to $1$ in $B$ and so $f_{j}(\omega)\neq 0$, for all $\omega\in B$
 and we have the thesis.
  
 \end{proof}

\begin{remark}
 The proof of the previous statement works also to prove that a slice regular function 
 $f:\Omega_D\rightarrow \mathbb{H}$ that is injective on a semi-slice 
 $D_J^+\subset \Omega_D$ has slice derivative nonzero over the same semi-slice $D_J^+$. We choose to formalize the theorem in the previous less general hypothesis 
 only to simplify the reading. 
\end{remark}

\begin{theorem}\label{maintheorem}
 Let $f$ be an injective slice regular function, then $N_f=\emptyset$.
\end{theorem}
\begin{proof}
If, by contradiction, there exists $x_0=\alpha+J\beta\in N_f\neq\emptyset$, 
then, thanks to Theorem \ref{thmchar}, the function $f-f(x_0)$ must have multiplicity 
$n$ greater
or equal to 2 at $\mathbb{S}_{x_0}$. This means that,
\begin{equation*}
 f(x)-f(x_0)=(x-x_0)\cdot g(x),
\end{equation*}
with $g\in\mathcal{SR}(\Omega_D)$ such that $g(x_1)=0$ for some $x_1\in\mathbb{S}_{x_0}$. 
Since $f$ is injective, then $g(x_0)=\frac{\partial f}{\partial x}(x_0)\neq 0$ and
$g(x_0^c)=\partial_s f(x_0)\neq 0$, and so $x_1\neq x_0,x_0^c$.
Now, whereas we know the values of $g$ at $x_0$ and at $x_0^c$, we can apply the 
representation formula in Theorem \ref{representationtheorem} to analyze the behavior 
of $f$ on the sphere 
$\mathbb{S}_{x_0}$. 
The result is the following,
\begin{equation*}
 g(\alpha+I\beta)=\frac{1}{2}\left(\frac{\partial f}{\partial x}(x_0)+\partial_s f(x_0)-IJ\left(\frac{\partial f}{\partial x}(x_0)-\partial_s f(x_0)\right)\right),\quad\forall I\in\mathbb{S}.
\end{equation*}
So, if there exist $I\in\mathbb{S}$ such that $g(\alpha+I\beta)=0$, then,
\begin{equation*}
 \begin{array}{lrcl}
   & \displaystyle\frac{\partial f}{\partial x}(x_0)+\partial_s f(x_0) & = & IJ\left(\displaystyle\frac{\partial f}{\partial x}(x_0)-\partial_s f(x_0)\right)\\
   & & & \\
   \Leftrightarrow & \displaystyle\frac{\partial f}{\partial x}(x_0)(\partial_s f(x_0))^{-1}+1 & = & IJ\left(\displaystyle\frac{\partial f}{\partial x}(x_0)(\partial_s f(x_0))^{-1}-1\right)\\
   & & & \\
   \Leftrightarrow & \displaystyle\frac{\partial f}{\partial x}(x_0)(\partial_s f(x_0))^{-1} & = & -(1-IJ)^{-1}(1+IJ),
 \end{array}
\end{equation*}
with $I\neq J,-J$, but then, since for $I\neq \pm J$ the product $-(1-IJ)^{-1}(1+IJ)$ has a non zero real part\footnote{This can be viewed using the 'scalar-vector' notation.}
, then $\frac{\partial f}{\partial x}(x_0)(\partial_s f(x_0))^{-1}$ does not belong to $\mathbb{C}_J^\bot$ and this is in contradiction with 
Proposition \ref{rank}.
 \end{proof}

\begin{example}
 Let $J\in\mathbb{S}$ be a fixed imaginary unit and $f:\mathbb{H}\setminus\mathbb{R}\rightarrow\mathbb{H}$ be the slice regular function
 defined as,
 \begin{equation*}
  f(\alpha+I\beta)=(\alpha+I\beta)(1-IJ).
 \end{equation*}
This function is constructed, by means of the representation formula, to be equal to zero over the semi-slice $\mathbb{C}_{-J}^+$ and 
to be equal to $2x$ over the opposite semi-slice $\mathbb{C}_J^+$. 
What we want to show is that the restriction $f|_{\mathbb{H}\setminus (\mathbb{R}\cup\mathbb{C}_{-J}^+)}$
 is injective. This is trivial if we restrict the function to a semi-slice $\mathbb{C}_I^+$, so let $x_1=\alpha_1+\beta_1I_1\neq \alpha_2+\beta_2I_2=x_2$, with $I_1\neq I_2$, then
 \begin{equation*}
 \begin{array}{c}
  f(x_1)=f(x_2)\Leftrightarrow\\
  \Leftrightarrow x_1(1-I_1J)=x_2(1-I_2J)\Leftrightarrow \\
  \Leftrightarrow x_1I_1(I_1+J)=x_2I_2(I_2+J) \Leftrightarrow \\
  \Leftrightarrow(x_2I_2)^{-1}(x_1I_1)=-\frac{1}{c}(I_2+J)(I_1+J),
 \end{array}
 \end{equation*}
 where $c=||I_{1}+J||^{2}\neq 0$.
Translating the variables $x_{1}$, $x_{2}$ into their components, we obtain that, the last equality is equivalent to the following one:
\begin{equation*}
 -\frac{1}{\alpha_2^2+\beta_2^2}[-\beta_1\beta_2+\alpha_1\beta_2 I_1-\alpha_2\beta_1 I_2+\alpha_1\alpha_2I_1I_2]=-\frac{1}{c}[I_2I_1+I_2J+JI_1-1].
\end{equation*}
Now we can decompose the last equation into the system involving the real and imaginary parts as follows:

\begin{equation*}
 \left\{\begin{array}{l}
         \displaystyle\frac{c}{\alpha_2^2+\beta_2^2}[\beta_1\beta_2+\alpha_1\alpha_2I_2\cdot I_1]=1+I_1\cdot I_2+(I_1+I_2)\cdot J\\
         \\
         \displaystyle\frac{c}{\alpha_2^2+\beta_2^2}[\alpha_1\beta_2 I_1-\alpha_2\beta_1 I_2+\alpha_1\alpha_2I_2\wedge I_1]=I_2\wedge I_1+(I_2-I_1)\wedge J
        \end{array}
\right. 
\end{equation*}
where $I\cdot J$ and $I\wedge J$ denote the scalar and the vector products\footnote{Here we used the 'scalar-vector' notation.} respectively in $\mathbb{R}^3$.
We will work now on the second equation of the previous system.

After the scalar product of the equation by $I_2-I_1$, we obtain that 
\begin{equation*}
 \alpha_1\beta_2=-\alpha_2\beta_1.
\end{equation*}
Substituting $\alpha_1=-\frac{\beta_1}{\beta_2}\alpha_2$ and multiplying scalarly by $I_1+I_2$ it follows that
\begin{equation*}
 (I_2\wedge I_1)\cdot J=\frac{c}{2}\frac{\alpha_1\beta_2}{\alpha_2^2+\beta_2^2}.
\end{equation*}
Taking into account the previous results and multiplying scalarly by $J$ and then by $I_1$ (or $I_2$), and supposing $\alpha_{2}\neq 0$, we obtain the following two equalities:
\begin{equation*}
 (I_1+I_2)\cdot J=-\frac{1}{2}\left[1+c\frac{\alpha_2}{\beta_2}\frac{\alpha_2\beta_1}{\alpha_2^2+\beta_2^2}\right],\quad I_1\cdot I_2=-\frac{1}{2}.
\end{equation*}
Putting all these ingredients in the first equation of the system one obtain that:

\begin{equation*}
 \frac{c\beta_1}{\alpha_2^2+\beta_2^2}\left[\beta_2+\frac{\alpha_2^2}{2\beta_2}\right]=
  -\frac{1}{2}\frac{c\beta_1}{\alpha_2^2+\beta_2^2}\frac{\alpha_{2}^{2}}{\beta_{2}},
\end{equation*}
and this is possible if and only if $\beta_{2}^{2}=-\alpha_{2}^{2}$, which is absurd. 
If now $\alpha_{2}=0$, following the first part of the same argument, we obtain $\alpha_{1}=0$ and so,
\begin{equation}\label{casealpha}
-\frac{c\beta_{1}}{\beta_{2}}=I_2I_1+I_2J+JI_1-1.
\end{equation}
But then, the imaginary part of $I_2I_1+I_2J+JI_1$, that is $I_2\wedge I_1+I_2\wedge J+J\wedge I_1$, must vanishes. This implies that  $(I_2\wedge I_1)\cdot J=0$ i.e.: 
$J= AI_{1}+BI_{2}$, for some $A$ and $B$ real numbers both different from zero.
In this case equation \ref{casealpha} becomes $A+B+1-\frac{c\beta_{1}}{\beta_{2}}=(1+A+B)I_{1}I_{2}$ and so $I_{1}\wedge I_{2}=0$. The last equalities (since $I_{1}\neq I_{2}$), 
entails $I_{1}=-I_{2}$ but this would imply $\frac{\beta_{1}}{\beta_{2}}=0$ and this is not possible.

Since this function, with the proper restriction, is slice regular and injective then Theorem \ref{maintheorem} says that its real differential is always invertible. 
This fact could also be seen computing the slice and the spherical derivative. Indeed, since 
\begin{equation*}
\partial_{s}f(\alpha+I\beta)=\frac{\beta-\alpha J}{\beta},
\end{equation*}
is always different from zero, we need only to control that the product $\frac{\partial f}{\partial x}(\alpha+I\beta)(\partial_{s}f(\alpha+I\beta))^{-1}$
does not belong to $\mathbb{C}_{I}^{\bot}$. Now,
\begin{equation*}
\frac{\partial f}{\partial x}(\alpha+I\beta)(\partial_{s}f(\alpha+I\beta))^{-1}=(1-IJ)\left(\frac{\beta-\alpha J}{\beta}\right)^{-1}=\frac{\beta(1-IJ)(\beta+\alpha J)}{\beta^{2}+\alpha^{2}},
\end{equation*}
and so, whenever $I\neq -J$, the previous product has a nonzero  real part and so does not belong to $\mathbb{C}_{I}^{\bot}$.
The real differential of $f$ can be represented in a point $x=\alpha+I\beta$, by means of slice and spherical forms, as,
\begin{equation*}
 df(\alpha+I\beta)=d_{sl}x(1-IJ)+d_{sp}x\left(1-\frac{\alpha}{\beta}J\right).
\end{equation*}
\end{example}

\begin{remark}
 In examples \ref{exe2} and \ref{exe3} we have studied the properties of the 
 slice regular function $h$ defined on $\mathbb{H}\setminus\mathbb{R}$ as
 $h(x)=(x+j)\cdot(1-Ii)=x(1-Ii)+(1+Ii)j$, where $x=\alpha+I\beta$. We found that this function admits 
 two surfaces, namely $S_h$ and $\mathbb{C}_{-i}^+$, on which it takes the constant 
 values $0$ and $2j$, respectively. These two surfaces, moreover, 
 yields the singular set $N_h$ of $h$. We want to show now, that,
 on $(\mathbb{H}\setminus\mathbb{R})\setminus(S_h\cup\mathbb{C}_{-i}^+)$ the 
 function is injective. Take then $I$ to be equal to $Ai+Bj+Ck$ and 
 $x=\alpha+I\beta$, then
 the function $h$, decomposed in its components with respect to $1,i,j,k$,
 is equal to,
 \begin{equation*}
  h(x)=\alpha(A+1)-C+(\beta(A+1)+B)i+(\beta B-\alpha C+1-A)j+(\alpha B+\beta C)k.
 \end{equation*}
First of all, since we are excluding the semi-slice $\mathbb{C}_{-i}^+$, we get 
$(A+1)\neq 0$. Given $q=q_0+q_1i+q_2j+q_3k\notin \{0,2j\}$, we want to
compute, where it is possible, $h^-1(q)$. So, we impose the system,
\begin{equation}\label{sys1}
 \begin{cases}
  \alpha(A+1)-C=q_0\\
  \beta(A+1)+B=q_1\\
  \beta B-\alpha C+1-A=q_2\\
  \alpha B+\beta C=q_3,
 \end{cases}
\end{equation}
and, substituting $\alpha=(q_0+C)(1+A)^{-1}$ and $\beta=(q_1-B)(1+A)^{-1}$ in the last two
equations and imposing $A^2+B^2+C^2=1$, we obtain that, for any $q$ such that $q_0^2+q_1^2\neq 0$,
\begin{equation}\label{sys2}
 \begin{cases}
 A=(q_0^2+q_1^2-q_2^2-q_3^2)/||q||^2\\
 B=2(q_0q_3+q_1q_2)/||q||^2\\
 C=2(q_1q_3-q_0q_2)||q||^2.
 \end{cases}
\end{equation}
From the last system we get,
\begin{equation}\label{sys3}
 \begin{cases}
 \alpha=\displaystyle\frac{q_0||q||^2+2(q_1q_3-q_0q_2)}{2(q_0^2+q_1^2)}\\
 \beta=\displaystyle\frac{q_1||q||^2-2(q_0q_3+q_1q_2)}{2(q_0^2+q_1^2)}.
 \end{cases}
\end{equation}
If $q_0^2+q_1^2=0$,then $q_0=0=q_1$ and then we get, from \ref{sys1}, that $q_3=0$ as well.
But then again, substituting $C=\alpha(A+1)$ and $B=-\beta(A+1)$ in the third equation and
imposing $A^2+B^2+C^2=1$, we obtain that, the only possibility are $q_2=0$ or $q_2=2$.
At the end, what we get is that, the image of 
$h$ is described as,
\begin{equation*}
 Im(h)=\left\{q\in\mathbb{H}\,|\,q_0^2+q_1^2\neq0\,q_1>\frac{2(q_0q_3+q_1q_2)}{||q||^2}\right\}\cup\{0,2j\},
\end{equation*}
moreover, since, for any $q\in Im(h)\setminus\{0,2j\}$ we can find only one 
preimage $h^{-1}(q)=\alpha+I\beta$, given by the two systems in equations \ref{sys2} and \ref{sys3}, 
then $\tilde{h}=h|_{(\mathbb{H}\setminus\mathbb{R})\setminus(S_h\cup\mathbb{C}_{-i}^+)}$
is injective and so $N_{\tilde{h}}=\emptyset$.
Its real differential can be expressed in a point $x=\alpha+I\beta$, again by means
of slice and spherical forms, as,
\begin{equation*}
 d\tilde{h}(\alpha+I\beta)=d_{sl}x(1-Ii)+d_{sp}x\left(1-\frac{\alpha}{\beta}i+\frac{k}{\beta}\right).
\end{equation*}
\end{remark}

\begin{remark}
 The reader could ask why we didn't follow the way of proving Theorem \ref{maintheorem} by Gentili, Salamon and Stoppato in \cite{gensalsto}. The answer is 
 that, of course, that proof doesn't work in the case in which the domain of the function does not have real points.
 This fact, rather than being a mere observation, give space to interesting considerations that are not studied in this paper.
 To be precise, the theorem that fails is the following:
 \begin{quote}
 \begin{theorem}
  Let $f:\Omega_D\rightarrow\mathbb{H}$ be a non-constant regular function, and let $\Omega_D\cap\mathbb{R}\neq\emptyset$. For each $x_0=\alpha+I\beta\in N_f$, there 
  exists a $n>1$, a neighborhood $U$ of $x_0$ and a neighborhood $T$ of $\mathbb{S}_{x_0}$ such that for all $x_1\in U$, the sum of the total multiplicities 
  of the zeros of $f-f(x_1)$ in $T$ equals n.
 \end{theorem}
 \end{quote}
 A counter example, if the domain does not have real points, is given by the function,
 \begin{equation*}
  \begin{array}{c}
   f:\mathbb{H}\setminus\mathbb{R}\rightarrow \mathbb{H}\\
   \alpha+I\beta\mapsto (\alpha+I\beta)(1-IJ),
  \end{array}
 \end{equation*}
for a fixed $J\in\mathbb{S}$. As we have seen, this function is injective over $\mathbb{H}\setminus (\mathbb{R}\cup\mathbb{C}_{-J}^+)$, and so, if we take $x_0=-J\in N_f$,
for any neighborhood $U$ of $-J$ and any neighborhood $T$ of $\mathbb{S}_{-J}$ the sum of total multiplicities of the zeros of $f-f(x_1)$, for any 
$x_1\in U\setminus \mathbb{C}_{-J}^+$ is equal to 1.
The previous function is constructed to be equal to 0 over $\mathbb{C}_{-J}^+$ and equal to $2x$ over $\mathbb{C}_{J}^+$, but other more complex examples can be build in
this way, for example considering a function equal to some monomial $x^m$ on a semi-slice and equal to another different monomial $x^n$ on the opposite.
This feature will certainly be a starting point for future investigations.
\end{remark}

\section{acknowledgement}
The present work was developed while I was a Ph.D. student in Mathematics at the University of Trento.
For this reason I want to thank my former institution and especially my supervisor prof. Alessandro Perotti.
Furthermore I was partially supported by the Project  FIRB ``Geometria Differenziale e Teoria Geometrica delle Funzioni'' and by GNSAGA of INdAM.

\end{document}